\documentclass[10pt,a4paper]{article}
\usepackage{xcolor}
\usepackage{comment}
\usepackage[utf8]{inputenc}
\usepackage[T1]{fontenc}
\usepackage[english]{babel}
\usepackage{lmodern}
\usepackage{microtype}
\usepackage{amsmath, amssymb, amsthm, mathtools, dsfont}
\usepackage[margin=1.5cm]{geometry}
\usepackage{xcolor}
\usepackage{enumitem}
\usepackage{hyperref} 

\numberwithin{equation}{section}

\theoremstyle{plain}
\newtheorem{theorem}{Theorem}[section]
\newtheorem{lemma}[theorem]{Lemma}
\newtheorem{proposition}[theorem]{Proposition}
\newtheorem{corollary}[theorem]{Corollary}

\theoremstyle{definition}
\newtheorem{definition}[theorem]{Definition}
\newtheorem{assumption}[theorem]{Assumption}

\theoremstyle{remark}
\newtheorem{remark}[theorem]{Remark}

\newcommand{\R}{\mathbb{R}}
\newcommand{\E}{\mathbb{E}}
\newcommand{\F}{\mathbb{F}}
\newcommand{\Prob}{\mathbb{P}}
\newcommand{\ind}{\mathds{1}}
\DeclareMathOperator{\Tr}{Tr}

\begin{document}

\title{Continuous Differentiability of the Value Function for Infinite-Dimensional Finite-Horizon Optimal Stopping and Related Variational Inequalities}
\author{Gabriele Bolli \footnote {Department of Mathematics "Guido Castelnuovo", Sapienza University of Rome, Rome, Italy. Email: gabriele.bolli@uniroma1.it.} \quad Giorgio Ferrari\footnote{Bielefeld University, Center for Mathematical Economics (IMW), Bielefeld , Germany. Email: giorgio.ferrari@uni-bielefeld.de}}
\date{\today}
\maketitle

\begin{abstract}
This paper studies finite-horizon optimal stopping problems for semilinear stochastic evolution equations in real, separable Hilbert spaces, together with their associated parabolic variational inequalities. We prove continuous differentiability of the value function in infinite dimensions, thereby obtaining a smooth-fit principle for the corresponding optimal stopping problem. The analysis has two parts. First, we prove existence and uniqueness, in a suitable weighted class, of a mild solution to the variational inequality and identify it with the optimal stopping value function. We also establish local spatial Lipschitz continuity of the value function by probabilistic methods, without requiring any smoothing property of the underlying transition semigroup. Second, under a global regularizing assumption on this semigroup, we prove higher-order spatial regularity: the value function is continuously Fr\'echet differentiable and its gradient is locally $\beta$-H\"older continuous. The abstract results are then applied to a stochastic heat equation with additive noise.
\end{abstract}

\textbf{Keywords}:
Optimal stopping, variational inequalities, infinite dimensional analysis, mild solutions, penalization methods, continuous differentiability, smooth-fit property.

\smallskip \noindent

\textbf{AMS classification}:
60G40, 49J40, 60H15, 35R15, 93E20.

\smallskip \noindent

\textbf{Acknowledgements}:
We thank Fausto Gozzi for comments and inspiring discussions. This work has been developed while Gabriele Bolli was visiting the Center for Mathematical Economics (IMW) at Bielefeld University. The hospitality of IMW is greatly acknowledged by the first named author. 

\section{Introduction}
Optimal stopping problems for stochastic evolution equations arise naturally when the underlying Markov state takes values in an infinite-dimensional space, for instance when it describes the spatial profile of a solution to a stochastic partial differential equation. In this Hilbert-space setting, the classical link between optimal stopping, free-boundary problems, and variational inequalities becomes substantially more delicate. The absence of a canonical Lebesgue measure on Hilbert spaces, the possible unboundedness of the data, and the limited smoothing produced by the transition semigroup prevent a direct use of finite-dimensional arguments. This paper studies these issues for semilinear stochastic dynamics on a real, separable Hilbert space and proves continuous differentiability of the associated value function.

Our contribution is twofold. First, under polynomial growth and local Lipschitz assumptions on the costs and the obstacle, we construct a mild solution of the parabolic variational inequality and prove uniqueness in a weighted class. The solution is obtained as the pointwise limit of penalized equations and is identified with the value function of the optimal stopping problem. In addition, a probabilistic argument yields local spatial Lipschitz continuity without imposing any smoothing condition on the state dynamics. Second, when the transition semigroup enjoys a global regularizing property, we derive higher-order spatial estimates for the penalized equations and pass them to the limit. This proves that the value function is continuously Fréchet differentiable and that its gradient is locally $\beta$-Hölder continuous. In optimal stopping terms, this regularity provides the corresponding smooth-fit property in infinite dimensions.

\paragraph{Our Framework.}
Let $(\Omega, \mathcal{F}, \mathbb{F} = \{\mathcal{F}_s\}_{s\geq0}, \mathbb{P})$ be a filtered probability space satisfying the usual conditions and carrying a cylindrical Wiener process $W$ with values in a separable Hilbert space $\Xi$. The state process evolves in a real, separable Hilbert space $H$ and solves, for $s \in (t,T]$, the semilinear stochastic evolution equation
\begin{equation*}
\mathrm{d}X(s) = [A X(s) + b(s, X(s))] \mathrm{d}s + \sigma(s, X(s)) \mathrm{d}W(s),
\end{equation*}
with initial condition $X(t)=x\in H$. The operator $A$ generates a strongly continuous semigroup, while the drift $b$ and the diffusion $\sigma$ satisfy the usual spatial Lipschitz and growth assumptions ensuring well-posedness of the stochastic flow.

For an initial state $(t,x) \in [0,T] \times H$, admissible stopping rules are the $\mathbb{F}$-stopping times in $\mathcal{T}_{t,T}$, taking values in $[t,T]$. Given a running cost $f$, a dynamic obstacle $\Phi$, and a terminal payoff $\Psi$, the cost associated with $\tau\in\mathcal{T}_{t,T}$ is
\begin{equation*}
J(t, x; \tau) := \mathbb{E} \left[ \int_t^\tau f(s, X(s)) , \mathrm{d}s + \Phi(\tau, X(\tau))\mathds{1}_{{\tau < T}} + \Psi(X(T))\mathds{1}_{{\tau = T}} \right].
\end{equation*}
The data are allowed to have polynomial growth and are assumed to be locally Lipschitz continuous in the spatial variable. The value function is therefore
\begin{equation*}
v(t, x) := \inf_{\tau \in \mathcal{T}_{t, T}} J(t, x; \tau).
\end{equation*}
The dynamic programming principle suggests that, whenever sufficient regularity is available, the value function should solve the parabolic variational inequality
\begin{equation*}
\max \left( -\partial_t V(t, x) - \mathcal{L} V(t, x) - f(t, x) , V(t, x) - \Phi(t, x) \right) = 0,
\end{equation*}
where $\mathcal{L}$ is the differential operator associated with the dynamics. Since classical differentiability is not available at the outset, the first part of the paper develops a mild formulation of this variational inequality and proves that its exact mild solution coincides with $v$.

\paragraph{Literature Overview.}The infinite-dimensional theory of optimal stopping and variational inequalities has been developed through several complementary approaches. In the viscosity framework, parabolic optimal stopping for Hilbert space-valued diffusions is studied in \cite{GaSw99}, where the value function is characterized as the unique viscosity solution of the corresponding variational inequality. This method accommodates irregular obstacles, but it does not provide the continuous differentiability of the value function that is needed to formulate smooth fit. A smooth-fit result in an elliptic infinite-dimensional setting is obtained in \cite{FeFeRiRo26}, under structural assumptions including a constant obstacle and a $\mathcal{C}^1$, semiconcave running cost with bounded gradient. Functional-analytic formulations, such as those developed in \cite{ChMe89} and \cite{BaSr06}, treat variational inequalities within $L^2$ spaces with respect to a given reference Gaussian measure. While successfully extending classical methods, they yield weak solutions defined only almost everywhere with respect to the reference measure, thereby precluding a comprehensive analysis of their classical spatial regularity. In \cite{ChDeA16}, Yosida approximations and finite-dimensional Galerkin projections are combined to construct approximating stopping problems and prove convergence to the value function solving the infinite-dimensional variational inequality, although spatial regularity is not addressed. A similar approximation procedure is used in \cite{MaMi08}, to characterize the value function of the optimal stopping problem as the solution to an infinite-dimensional variational inequality; however, the subsequent analysis focuses primarily on numerical approximations, leaving the spatial regularity of the solution unaddressed. Strong and mild solutions to variational inequalities and Bellman differential inclusions are studied in \cite{BaMa08} and \cite{Za01} by means of excessive measures associated with a transition semigroup; in that setting, the link between the variational inequality and the value function is obtained, again, almost surely with respect to the reference measure, rather than pointwise. From a purely topological perspective, optimal stopping for general Markov processes on Polish spaces is investigated in \cite{Za97} using time-discretization. Although this framework establishes the continuity of the value function with remarkable generality, the absence of a linear and differential structure in the underlying metric space precludes the derivation of higher-order spatial regularity. A probabilistic approach based on reflected backward stochastic differential equations is developed in \cite{FuMaTe17}. There, with a notion of mild solution close to the one used here, the value function is identified as the unique mild solution and local Lipschitz continuity is obtained, but higher-order differentiability remains outside the scope of the analysis.The present paper combines a pointwise penalization method with semigroup regularization estimates. Following the analysis in the finite-dimensional setting developed in \cite[Chapter 3]{BeLi82}, we approximate the variational inequality by penalized equations. Unlike the finite-dimensional Sobolev-space approach of \cite{BeLi82}, the absence of Lebesgue measure in Hilbert spaces leads us to work with mild integral formulations and pointwise weighted estimates. The continuous differentiability result is then obtained by exploiting regularizing properties of the transition semigroup rather than imposing strong differentiability assumptions on all problem data. This identifies a useful trade-off: compared with viscosity and BSDE methods, higher differentiability of the value function is recovered under a global smoothing assumption on the stochastic flow.

\paragraph{Our Methods and Results.}
Inspired by \cite{BeLi82}, we introduce the following family of penalized problems:
\begin{equation*}
\begin{cases}
\partial_t u_\varepsilon(t, x) + \mathcal{L} u_\varepsilon(t, x) + f(t, x) - \frac{1}{\varepsilon} \left( u_\varepsilon(t, x) - \Phi(t, x) \right)^+ = 0, & \forall (t, x) \in [0, T) \times H, \\
u_\varepsilon(T, x) = \Psi(x), & \forall x \in H.
\end{cases}
\end{equation*}
In this equation, the penalty term $\frac{1}{\varepsilon} \left( u_\varepsilon(t, x) - \Phi(t, x) \right)^+$ acts as a relaxation of the constraint $V(t,x)\leq \Phi(t,x)$. When $u_\varepsilon$ is below the obstacle, the dynamics coincide with the unconstrained dynamics of the variational inequality. When $u_\varepsilon$ exceeds $\Phi$, the penalty term forces the solution back towards the admissible region, and this effect becomes stronger as $\varepsilon \to 0$. Thus, the penalized family provides a natural approximation of the variational inequality. In contrast to \cite{BeLi82}, which relies on weak formulations in Sobolev spaces and derives its main estimates through $L^p$-bounds on the penalty term, the lack of an infinite-dimensional analogue of Lebesgue measure requires us to work with mild integral formulations and purely pointwise weighted estimates.

In the first part of the paper, after proving existence and uniqueness of a penalized solution for every $\varepsilon > 0$ (Theorem \ref{thm:existence_uniqueness}), we establish the monotonicity of the solutions $u_\varepsilon$ (Proposition \ref{prop:penalty_monotonicity}), as well as their pointwise boundedness (Theorem \ref{thm:uniform_lower_bound}). Consequently, we deduce that this family admits a pointwise limit, and that this limit coincides with the exact mild solution of the variational inequality (Theorem \ref{thm:convergence_exact}). Furthermore, we prove that this solution is locally Lipschitz continuous (Corollary \ref{cor:value_function_lipschitz}). To obtain estimates independent of the parameter $\varepsilon>0$, we identify the penalized solution with the value of an auxiliary minimization problem (Theorem \ref{thm:feynman_kac}), in analogy with the finite-dimensional argument in \cite[Chapter 3, Section 4, Theorem 4.4]{BeLi82}. The corresponding cost functional is
\begin{equation*}
J_\varepsilon(t, x; \gamma) := \mathbb{E} \left[ \int_t^T e^{-\int_t^s \gamma_r \mathrm{d}r} \Big( f(s, X(s)) + \gamma_s \Phi(s, X(s)) \Big) \mathrm{d}s + e^{-\int_t^T \gamma_r \mathrm{d}r} \Psi(X(T)) \right],
\end{equation*}
where $\gamma$ is a progressively measurable process taking values in $[0,1/\varepsilon]$. The final step of this first part consists in identifying the mild solution of the variational inequality with the value function of the optimal stopping problem and in deriving an optimal stopping time (Theorem \ref{thm:optimal_strategy}). To this end, the mild structure of the penalized equations is fundamental: the action of the transition semigroup provides the tools to construct suitable martingales, used tyo identify the optimal stopping time as the first entry time of the state process into the stopping region, where the value function coincides with the obstacle.

In the second part of the work, we focus on the spatial regularity of the value function. We assume that the transition semigroup has a quantitative smoothing property: for every continuous function $\phi$ with polynomial growth, the mapping $P_{t,s}[\phi]$ is twice Fréchet differentiable, with first- and second-order derivatives controlled by singular powers of $s-t$ (Assumption \ref{ass:smoothing}). The order of the singularity determines the admissible range of Hölder exponents $\beta$ for which the temporal convolutions in the mild formulation yield spatial $\mathcal{C}^{1,\beta}_{\mathrm{loc}}$-regularity. Exploiting this regularizing mechanism, we first establish spatial gradient estimates for the penalized solutions $u_\varepsilon$ (Lemma \ref{lem:global_reg_approx}). Then, by ensuring that these bounds are uniform with respect to the penalization parameter $\varepsilon$, we can rigorously pass to the limit, proving that the value function of the optimal stopping problem is continuously Fréchet differentiable with a locally $\beta$-Hölder continuous gradient (Theorem \ref{thm:frechet_ident}). In optimal stopping terms, this yields a smooth-fit principle in infinite dimensions.

A central technical point in this passage to the limit is the control of the penalty term
$
\Pi_\varepsilon(t,x):=\frac{1}{\varepsilon}\big(u_\varepsilon(t,x)-\Phi(t,x)\big)^+ .
$
To prove the convergence of the penalized equations to the exact mild solution of the variational inequality, to identify this limit with the optimal stopping value through a hitting-time verification argument, and to transfer the spatial regularity estimates from the penalized problems to the limiting value function, one needs a bound on $\Pi_\varepsilon$ which is uniform in $\varepsilon$. We isolate this requirement in the notion of an admissible obstacle (Definition \ref{def:admissible_obstacle}). This condition is not meant as an additional modelling assumption on the stopping problem, but rather as the precise compatibility requirement between the obstacle, the dynamics, and the penalization procedure. Section \ref{sec:admissible_obstacles} provides checkable sufficient conditions for admissibility, including regular obstacles, quadratic-type obstacles, and cylindrical concave obstacles.

Finally, in Section \ref{sec:example_heat}, we apply the abstract theory to a stochastic heat equation. This example shows that the smoothing assumption used in the regularity analysis is compatible with a genuine infinite-dimensional SPDE driven by suitably chosen trace-class noise.

\section{Abstract Framework}
In this section we introduce the functional and probabilistic framework required to analyze optimal stopping problems in infinite-dimensional spaces. By combining the framework of weighted functional spaces with the analysis of semilinear stochastic differential equations, we establish the rigorous foundation necessary for the subsequent investigation of the value function's characterization and regularity.

\subsection{Functional Setting and Weighted Topology}
\label{sec:functional_setting}

We define the functional framework for the finite-horizon optimal stopping problem. The state space is endowed with a weighted topology to control unbounded variables and admit problem data with polynomial growth.

\begin{definition}[Base Hilbert Spaces]
\label{def:base_spaces}
Let $(H, \langle \cdot, \cdot \rangle_H, \|\cdot\|_H)$ and $(\Xi, \langle \cdot, \cdot \rangle_\Xi, \|\cdot\|_\Xi)$ be real, separable Hilbert spaces equipped with their respective Borel $\sigma$-algebras $\mathcal{B}(H)$ and $\mathcal{B}(\Xi)$. We refer to $H$ as the \emph{state space} and $\Xi$ as the \emph{noise space}.
\end{definition}

\begin{definition}[Weighted Polynomial Topologies]
\label{def:weighted_topology}
Let $m \ge 1$. We define the continuous \emph{spatial weight function} $w_m \colon H \to \mathbb{R}$ by
\begin{equation} 
\label{eq:weight_function}
    w_m(x) := 1 + \|x\|_H^m, \quad \forall x \in H.
\end{equation}
Associated with this weight, we introduce the following functional spaces:
\begin{enumerate}[label={(\textit{\roman*})}, leftmargin=*]
    \item The \emph{measurable weighted space} $\mathcal{B}_m(H)$ consists of all Borel-measurable mappings $\varphi \colon H \to \mathbb{R}$ such that the following norm is finite:
    \begin{equation}
        \|\varphi\|_{\mathcal{B}_m} := \sup_{x \in H} \frac{|\varphi(x)|}{w_m(x)}.
    \end{equation}

    \item The \emph{continuous weighted space} $\mathcal{C}_m(H)$ is the subspace of $\mathcal{B}_m(H)$ consisting of mappings $\varphi$ that are continuous on $H$. It is equipped with the same norm $\|\cdot\|_{\mathcal{B}_{m}}$.
    
    \item The \emph{parabolic measurable weighted space} $\mathcal{B}_m([0, T] \times H)$ consists of all Borel-measurable mappings $\psi \colon [0, T] \times H \to \mathbb{R}$ such that the following norm is finite:
    \begin{equation}
        \|\psi\|_{\mathcal{B}_{m, T}} := \sup_{(t, x) \in [0, T] \times H} \frac{|\psi(t, x)|}{w_m(x)}.
    \end{equation}

    \item The \emph{parabolic continuous weighted space} $\mathcal{C}_m([0, T] \times H)$ is the subspace of $\mathcal{B}_m([0, T] \times H)$ consisting of mappings $\psi$ that are jointly continuous on $[0, T] \times H$. It is equipped with the same norm $\|\cdot\|_{\mathcal{B}_{m, T}}$.
    
    \item The \emph{spatial weighted local Lipschitz space} $\mathrm{Lip}_m(H)$ is the subspace of $\mathcal{C}_m(H)$ consisting of mappings $\varphi \colon H \to \R$ for which there exists a constant $K > 0$ such that
    \begin{equation}
        |\varphi(x) - \varphi(y)| \le K \big(1 + \|x\|_H^{m-1} + \|y\|_H^{m-1}\big) \|x - y\|_H
    \end{equation}
    for all $x, y \in H$.
  
    \item The \emph{weighted local Lipschitz space} $\mathrm{Lip}_{m}([0, T] \times H)$ is the subspace of $\mathcal{C}_m([0, T] \times H)$ consisting of mappings $\psi$ for which there exists a constant $K > 0$ such that
    \begin{equation} 
        |\psi(t, x) - \psi(t, y)| \le K \big(1 + \|x\|_H^{m-1} + \|y\|_H^{m-1}\big) \|x - y\|_H
    \end{equation}
    for all $x, y \in H$ and $t \in [0, T]$. 
\end{enumerate}
\end{definition}

The following Lemma ensures that the parabolic measurable weighted space $\mathcal{B}_m([0, T] \times H)$ is a Banach space.

\begin{lemma}[Completeness of the Parabolic Measurable Weighted Space]
\label{lem:completeness_weighted_space}
The space $\mathcal{B}_m([0, T] \times H)$ endowed with the norm $\|\cdot\|_{\mathcal{B}_{m, T}}$ is a Banach space.
\end{lemma}

\begin{proof}
Let $\{\psi_n\}_{n \in \mathbb{N}}$ be a Cauchy sequence in $\mathcal{B}_m([0, T] \times H)$. For any $\varepsilon > 0$, there exists $N \in \mathbb{N}$ such that: 
$$|\psi_n(t, x) - \psi_k(t, x)| \le \varepsilon w_m(x)$$ 
for all $n, k \ge N$ and $(t,x)\in [0,T]\times H$. 
Since the weight $w_m(x)$ is finite for every $x \in H$, for any fixed point $(t,x)$, the sequence of real numbers $\{\psi_n(t,x)\}_{n \in \mathbb{N}}$ is a Cauchy sequence in $\mathbb{R}$. By the completeness of $\mathbb{R}$, the sequence converges pointwise to a limit function $\psi(t,x) := \lim_{n \to \infty} \psi_n(t,x)$. Since each $\psi_n$ is a Borel-measurable mapping, their pointwise limit $\psi$ is automatically Borel-measurable on $[0, T] \times H$. Taking the limit as $k \to \infty$ in the initial Cauchy bound yields: 
$$|\psi_n(t, x) - \psi(t, x)| \le \varepsilon w_m(x)$$ 
for all $n \ge N$ and $(t,x)\in [0,T]\times H$. Taking the supremum over $[0, T] \times H$ proves $\|\psi_n - \psi\|_{\mathcal{B}_{m, T}} \to 0$. Finally, the triangle inequality $\|\psi\|_{\mathcal{B}_{m,T}} \le \|\psi - \psi_N\|_{\mathcal{B}_{m,T}} + \|\psi_N\|_{\mathcal{B}_{m,T}} < \infty$ ensures that the limit $\psi$ preserves the weighted polynomial growth, proving that $\psi \in \mathcal{B}_m([0, T] \times H)$.
\end{proof}

\begin{lemma}[Closedness of the Continuous Weighted Subspace]
\label{lem:closedness_continuous_space}
The parabolic continuous weighted space $\mathcal{C}_m([0, T] \times H)$ is a closed subspace of $\mathcal{B}_m([0, T] \times H)$. Consequently, endowed with the norm $\|\cdot\|_{\mathcal{B}_{m, T}}$, it is a Banach space.
\end{lemma}

\begin{proof}
Let $\{\psi_n\}_{n \in \mathbb{N}} \subset \mathcal{C}_m([0, T] \times H)$ be a sequence converging to a limit $\psi \in \mathcal{B}_m([0, T] \times H)$ in the weighted topology, meaning $\|\psi_n - \psi\|_{\mathcal{B}_{m, T}} \to 0$ as $n \to \infty$. To prove that $\mathcal{C}_m([0, T] \times H)$ is closed, it suffices to show that the limit mapping $\psi$ is jointly continuous on $[0, T] \times H$. Fix an arbitrary radius $R > 0$ and consider the closed ball $B_R := \{x \in H : \|x\|_H \le R\}$. On the cylindrical domain $[0, T] \times B_R$, the spatial weight function is strictly bounded by a constant:
$$w_m(x) = 1 + \|x\|_H^m \le 1 + R^m := C_R, \quad \forall x \in B_R.$$
For any $(t, x) \in [0, T] \times B_R$, we can bound the pointwise difference by extracting the weighted norm:
$$|\psi_n(t, x) - \psi(t, x)| \le \frac{|\psi_n(t, x) - \psi(t, x)|}{w_m(x)} w_m(x) \le \|\psi_n - \psi\|_{\mathcal{B}_{m, T}} C_R.$$
Taking the supremum over $[0, T] \times B_R$ yields:
$$\sup_{(t, x) \in [0, T] \times B_R} |\psi_n(t, x) - \psi(t, x)| \le C_R \|\psi_n - \psi\|_{\mathcal{B}_{m, T}}.$$
As $n \to \infty$, the right-hand side vanishes, proving that the sequence of continuous mappings $\{\psi_n\}$ converges to $\psi$ \emph{uniformly} on the bounded domain $[0, T] \times B_R$. Since the uniform limit of continuous functions is continuous, $\psi$ is jointly continuous on $[0, T] \times B_R$. Because the radius $R > 0$ is arbitrary, every point in $[0, T] \times H$ admits a bounded neighborhood where $\psi$ is continuous. Thus, $\psi$ is jointly continuous on $[0, T] \times H$. By definition, this ensures $\psi \in \mathcal{C}_m([0, T] \times H)$, concluding the proof.
\end{proof}

\subsection{Stochastic Dynamics}
\label{sec:dynamics}
 Throughout the discussion, we fix a filtered probability space $(\Omega, \mathcal{F}, \F = \{\mathcal{F}_s\}_{s \ge 0}, \Prob)$ satisfying the usual conditions (i.e., the filtration $\F$ is right-continuous and $\mathcal{F}_0$ is augmented with all $\Prob$-null sets). Let $W = \{W(s)\}_{s \ge0}$ denote a cylindrical $\F$-Wiener process (see, e.g., \cite[Definition 1.88]{FaGoSw17}) defined on an auxiliary real, separable Hilbert space $\Xi$. We require the diffusion coefficient to take values in $\mathcal{L}_2(\Xi, H)$, the space of Hilbert-Schmidt operators mapping $\Xi$ into $H$. For $t \in [0, T)$ and $x \in H$, the system solves the abstract Cauchy problem:
\begin{equation} 
\label{eq:cauchy_problem}
    \begin{cases}
        \mathrm{d}X(s) = [A X(s) + b(s, X(s))] \mathrm{d}s + \sigma(s, X(s)) \mathrm{d}W(s), \quad s \in (t, T], \\
        X(t) = x \in H.
    \end{cases}
\end{equation}
To ensure the well-posedness of the stochastic evolution, the linear generator $A$ and the measurable coefficients $b$ and $\sigma$ are assumed to satisfy the following assumptions.
\begin{assumption}[Structural Conditions on Operators and Coefficients]
\label{ass:coefficients}
\begin{enumerate}[label={(\textit{\roman*})}, leftmargin=*]
\item[]
    \item \emph{Infinitesimal Generator.} The operator $A \colon \mathrm{Dom}(A) \subset H \to H$ is a closed, densely defined linear operator that generates a strongly continuous ($C_0$) semigroup $\{e^{\tau A}\}_{\tau \ge 0}$ on $H$. \footnote{It is well known (see, e.g., \cite[Chapter1, Theorem 2.2]{Pa83}) that there exist constants $\omega \in \mathbb{R}$ and $M\geq1$ such that $\|e^{\tau A}\|_{\mathcal{L}(H)} \le Me^{\omega \tau}, \quad \forall \tau \ge 0.$}
    
    \item \emph{Lipschitz Continuity.} The mappings $b \colon [0, T] \times H \to H$ and $\sigma \colon [0, T] \times H \to \mathcal{L}_2(\Xi, H)$ are Borel measurable. Moreover, there exists a constant $L > 0$ such that:
    \begin{equation}
        \|b(s, x) - b(s, y)\|_H + \|\sigma(s, x) - \sigma(s, y)\|_{\mathcal{L}_2(\Xi, H)} \le L \|x - y\|_H,
    \end{equation}
    for all $s \in [0, T]$ and all $x, y \in H$.
    
    \item \emph{Linear Growth.} The coefficients exhibit at most linear growth in the spatial variable. That is, there exists a constant $K > 0$ such that:
    \begin{equation}
        \|b(s, x)\|_H + \|\sigma(s, x)\|_{\mathcal{L}_2(\Xi, H)} \le K \big( 1 + \|x\|_H \big),
    \end{equation}
    for all $(s, x) \in [0, T] \times H$.
\end{enumerate}
\end{assumption}
Under these conditions (Assumption \ref{ass:coefficients}), we introduce the notion of mild solution.

\begin{definition}[Mild Solution]
\label{def:mild_solution}
For an initial condition $t \in [0, T)$ and $x \in H$, an $H$-valued, $\F$-progressively measurable process $X(\cdot) \equiv X(\cdot; t, x)$ is a \emph{mild solution} to the abstract Cauchy problem \eqref{eq:cauchy_problem} if it satisfies $\Prob$-a.s. for all $s \in [t, T]$:
\begin{equation} 
\label{eq:mild_solution_eq}
    X(s) = e^{(s-t)A}x + \int_t^s e^{(s-r)A} b(r, X(r)) \, \mathrm{d}r + \int_t^s e^{(s-r)A} \sigma(r, X(r)) \, \mathrm{d}W(r).
\end{equation}
\end{definition}
\begin{remark}
    By Assumption \ref{ass:coefficients} this mild solution exists, is unique up to indistinguishability, and has a modification with $\Prob$-a.s. continuous sample paths, as shown in \cite[Theorem 5.3.1]{DPZa96}. 
\end{remark}
The stochastic flow satisfies the following moment estimates.

\begin{lemma}[Parabolic Moment Estimates] 
\label{lem:moment_estimates}
Let Assumption \ref{ass:coefficients} hold. For any real exponent $m\ge2$, there exists a constant $K_m > 0$ such that the mild solution $X(\cdot; t, x)$ introduced in Definition \ref{def:mild_solution} satisfies the inequality:
\begin{equation} 
\label{eq:moment_estimates_eq}
    \E \left[ \sup_{s \in [t, T]} \|X(s; t, x)\|_H^m \right] \le K_m \left( 1 + \|x\|_H^m \right),
\end{equation}
for all initial conditions $(t, x) \in [0, T] \times H$.
\end{lemma}

\begin{proof}
    This is a standard result. For a complete proof see, e.g., \cite[Theorem 7.2]{DPZa14}.
\end{proof}
We introduce the transition semigroup associated with the mild solution \eqref{eq:mild_solution_eq} of the abstract Cauchy problem \eqref{eq:cauchy_problem}.

\begin{definition}[Transition semigroup]
\label{def:evolution_operators}
Let $X(\cdot; t, x)$ denote the $H$-valued Markov process corresponding to the unique mild solution of the semilinear stochastic differential equation \eqref{eq:cauchy_problem}. The two-parameter transition semigroup $\{P_{t,s}\}_{0 \le t \le s \le T}$ acting on the measurable weighted space $\mathcal{B}_m(H)$ is defined pointwise by
\begin{equation}\label{eq:transition_semigroup}
    P_{t,s} [\varphi](x) := \mathbb{E} \big[ \varphi(X(s; t, x)) \big], \quad \forall \varphi \in \mathcal{B}_m(H).
\end{equation}
\end{definition}

\begin{remark}
The well-definedness of this expectation, and the fact that $P_{t,s}[\varphi]$ inherits the polynomial growth bound, follows directly from the uniform parabolic moment estimates (Lemma \ref{lem:moment_estimates}).
\end{remark}

\section{The Optimal Stopping Problem and its Penalization}
\noindent In this section we provide a rigorous formulation of the finite horizon optimal stopping problem, the associated value function and the variational inequality within the Hilbertian setting. We employ a penalization framework, approximating the exact mild solution of the variational inequality through a family of well-posed mild integral equations. We refer to \cite[Chapter 3]{BeLi82} for an analogous approach in the finite dimensional case, in the framework of Sobolev spaces.

\subsection{Problem Formulation}
\label{sec:optimal_stopping}

The expected cost functional is determined by a running cost, an early-stopping dynamic obstacle and a distinct terminal payoff. The following growth and regularity conditions are imposed on the problem data.

\begin{assumption}[Structural Regularity and Compatibility of Problem Data] 
\label{ass:problem_data}
\begin{enumerate}[label={(\textit{\roman*})}, leftmargin=*]
    \item[] 
    \item \emph{Regularity:} The running cost $f$ and the dynamic obstacle mapping $\Phi$ belong to the space $Lip_{m}([0, T] \times H)$. The terminal payoff $\Psi$ belongs to the space $Lip_m(H)$. Consequently, there exists a constant $L > 0$ such that, for all $t \in [0, T]$ and $x, y \in H$, the following estimate holds:
    \begin{align}
        |f(t, x) - f(t, y)|+|\Phi(t, x) - \Phi(t, y)| + |\Psi(x) - \Psi(y)| &\le L \big(1 + \|x\|_H^{m-1} + \|y\|_H^{m-1}\big) \|x - y\|_H.
    \end{align}
    \item \emph{Terminal Compatibility:} The terminal payoff is bounded from above by the dynamic obstacle at the terminal horizon $T$. Namely, for all $x \in H$:
    \begin{equation}
        \Psi(x) \le \Phi(T, x).
    \end{equation}
\end{enumerate}
\end{assumption}

We formulate the finite-horizon parabolic optimal stopping problem associated with the semilinear dynamics on the Hilbert space $H$. 

\begin{definition}[Expected Cost and Value Function]
\label{def:value_function}
For a given initial time $t \in [0, T]$, let $\mathcal{T}_{t, T}$ denote the class of all $\F$-stopping times $\tau \colon \Omega \to [t, T]$. The finite-horizon optimal stopping problem is characterized by the following quantities:
\begin{enumerate}[label={(\textit{\roman*})}, leftmargin=*]
    \item \emph{Expected Cost Functional.} For an initial state $x \in H$ and an admissible stopping strategy $\tau \in \mathcal{T}_{t, T}$, the expected cost functional $J \colon [0, T] \times H \times \mathcal{T}_{t, T} \to \R$ is defined by:
    \begin{equation} 
    \label{eq:def_cost}
        J(t, x; \tau) := \E \left[ \int_t^\tau f(s, X(s)) \, \mathrm{d}s + \Phi(\tau, X(\tau))\ind_{\{\tau < T\}} + \Psi(X(T))\ind_{\{\tau = T\}} \right],
    \end{equation}
    where $X(\cdot) \equiv X(\cdot; t, x)$ denotes the continuous mild solution of the governing semilinear state equation \eqref{eq:mild_solution_eq}.
    
    \item \emph{Value Function.} The value function $v \colon [0, T] \times H \to \R$ is defined as the infimum of the expected cost evaluated over the entire class of admissible stopping strategies:
    \begin{equation} 
    \label{eq:def_value}
        v(t, x) := \inf_{\tau \in \mathcal{T}_{t, T}} J(t, x; \tau).
    \end{equation}
\end{enumerate}
\end{definition}

The following Proposition ensures the well-posedness of the cost functional.
\begin{proposition}[Absolute Integrability]
\label{prop:abs_integrability}
Let Assumptions \ref{ass:coefficients} and \ref{ass:problem_data} hold. Then, the expected cost $J(t, x; \tau)$ is absolutely integrable for any initial configuration $(t, x) \in [0, T] \times H$ and any admissible stopping time $\tau \in \mathcal{T}_{t, T}$.
\end{proposition}

\begin{proof}
    This is a consequence of Assumption \ref{ass:coefficients} together with the moment estimate established in Lemma \ref{lem:moment_estimates}.
\end{proof}

We introduce the variational inequality associated to the optimal stopping problem.

\begin{definition}[Classical Solution to the Variational Inequality]
\label{def:formal_vi}
A mapping $V \in \mathcal{C}^{1,2}([0, T) \times H) \cap \mathcal{C}([0, T] \times H)$ is a \emph{classical solution} to the parabolic variational inequality if it satisfies:
\begin{enumerate}[label={(\textit{\roman*})}, leftmargin=*]
    \item \emph{Regularity.} $\nabla V(t, x) \in \mathrm{Dom}(A^*)$ for all $(t, x) \in [0, T) \times H$, and the map $(t, x) \mapsto A^* \nabla V(t, x)$ is continuous.
    \item \emph{Terminal condition.} $V(T, x) = \Psi(x)$ for all $x \in H$.
    \item \emph{Pointwise equation.} The variational inequality holds pointwise on $[0, T) \times H$:
    \begin{equation}
    \label{eq:variational_inequality}
        \max \left( -\partial_t V(t, x) - \mathcal{L} V(t, x) - f(t, x), \, V(t, x) - \Phi(t, x) \right) = 0,
    \end{equation}
\end{enumerate}
where the differential operator $\mathcal{L}$ is defined by:
\begin{equation}
\label{eq:formal_operator}
    \mathcal{L} V(t, x) := \langle x, A^* \nabla V(t, x) \rangle_H + \langle b(t, x), \nabla V(t, x) \rangle_H + \frac{1}{2} \Tr \big( \sigma(t, x) \sigma^*(t, x) \nabla^2 V(t, x) \big).
\end{equation}
\end{definition}

\subsection{Penalization Method and Mild Formulations}
\label{sec:yosida}

Following \cite[Chapter 3]{BeLi82} we now introduce a penalization scheme to handle the obstacle constraint in the finite-horizon optimal stopping problem. By relaxing the condition $V \le \Phi$ through the nonlinear operator $-\frac{1}{\varepsilon}(u_\varepsilon - \Phi)^+$, the original problem is approximated by a family of unconstrained backward equations. 

The penalized Hamilton-Jacobi-Bellman (HJB) terminal value problem is formulated as follows:
\begin{equation} 
\label{eq:hjb_penalized}
    \begin{cases}
        \partial_t u_\varepsilon(t, x) + \mathcal{L} u_\varepsilon(t, x) + f(t, x) - \frac{1}{\varepsilon} \left( u_\varepsilon(t, x) - \Phi(t, x) \right)^+ = 0, & \forall (t, x) \in [0, T) \times H, \\
        u_\varepsilon(T, x) = \Psi(x), & \forall x \in H,
    \end{cases}
\end{equation}
where $\mathcal{L}$ denotes the differential operator introduced in \eqref{eq:formal_operator}.

\begin{remark}[Formal Differential Operator]
\label{rem:formal_operator}
The differential formulation \eqref{eq:hjb_penalized} serves primarily a heuristic purpose. The unbounded nature of the operator $A$ preclude in general the existence of classical solutions in $H$. Consequently, the analysis must proceed through equivalent mild integral equations (see \cite[Chapter 4]{FaGoSw17} or \cite{DPZa02} for more details) driven by the transition semigroup.
\end{remark}

\begin{definition}[Mild Integral Formulation]
\label{def:mild_formulation}
A mapping $u_\varepsilon \in \mathcal{B}_m([0, T] \times H)$ is defined as a \emph{mild solution} to the penalized backward equation \eqref{eq:hjb_penalized} if it satisfies the following integral equation on $[0, T] \times H$:
\begin{equation} 
\label{eq:direct_formulation}
    u_\varepsilon(t, x) = P_{t,T}[\Psi](x) + \int_t^T P_{t,s} \left[ f(s, \cdot) - \frac{1}{\varepsilon} \left( u_\varepsilon(s, \cdot) - \Phi(s, \cdot) \right)^+ \right](x) \, \mathrm{d}s.
\end{equation}
\end{definition}
The following Theorem provides an equivalent formulation for the mild solution.
\begin{theorem}[Shifted Mild Formulation and Equivalence]
\label{thm:equivalence_mild}
Let Assumptions \ref{ass:coefficients} and \ref{ass:problem_data} hold. For any mapping $u_\varepsilon \in \mathcal{B}_m([0, T] \times H)$ and any shift parameter $L \ge 0$, $u_\varepsilon$ is a mild solution of equation \eqref{eq:hjb_penalized} if and only if it satisfies the \emph{shifted formulation}:
\begin{equation} 
\label{eq:shifted_formulation}
    u_\varepsilon(t, x) = e^{-L(T-t)} P_{t,T}[\Psi](x) + \int_t^T e^{-L(s-t)} P_{t,s} \left[ f(s, \cdot) + L u_\varepsilon(s, \cdot) - \frac{1}{\varepsilon} \left( u_\varepsilon(s, \cdot) - \Phi(s, \cdot) \right)^+ \right](x) \, \mathrm{d}s.
\end{equation}
\end{theorem}

\begin{proof}
We introduce the auxiliary mapping $H(s, \cdot) := f(s, \cdot) - \frac{1}{\varepsilon} (u_\varepsilon(s, \cdot) - \Phi(s, \cdot))^+$. By assumption, $H \in \mathcal{B}_m([0, T] \times H)$. Assume $u_\varepsilon$ satisfies the direct mild formulation. Substituting this into the right-hand side of \eqref{eq:shifted_formulation} defines the mapping:
\begin{align} 
\label{eq:proof_S_def}
    S(t, x) &:= e^{-L(T-t)} P_{t,T}[\Psi](x) + \int_t^T e^{-L(s-t)} P_{t,s} [H(s, \cdot)](x) \, \mathrm{d}s \nonumber \\ 
    &\quad + \int_t^T L e^{-L(s-t)} P_{t,s} \left[ P_{s,T}[\Psi] + \int_s^T P_{s,r} [H(r, \cdot)] \, \mathrm{d}r \right](x) \, \mathrm{d}s. 
\end{align}
By the semigroup property $P_{t,s} \circ P_{s,T} = P_{t,T}$, the integral involving the terminal condition simplifies to:
\begin{equation}
\label{eq:proof_boundary_eval}
    \int_t^T L e^{-L(s-t)} P_{t,T}[\Psi](x) \, \mathrm{d}s = P_{t,T}[\Psi](x) \left( 1 - e^{-L(T-t)} \right).
\end{equation}
Applying Fubini's Theorem to reverse the integration order in the double integral and using $P_{t,s} \circ P_{s,r} = P_{t,r}$ yields:
\begin{align}
\label{eq:proof_double_int}
    &\int_t^T \int_s^T L e^{-L(s-t)} P_{t,r} [H(r, \cdot)](x) \, \mathrm{d}r \, \mathrm{d}s = \int_t^T P_{t,r} [H(r, \cdot)](x) \left( 1 - e^{-L(r-t)} \right) \mathrm{d}r.
\end{align}
Combining \eqref{eq:proof_boundary_eval} and \eqref{eq:proof_double_int} with \eqref{eq:proof_S_def} cancels the exponential weights, recovering the direct mild formulation:
\begin{equation}
    S(t, x) = P_{t,T}[\Psi](x) + \int_t^T P_{t,s} [H(s, \cdot)](x) \, \mathrm{d}s = u_\varepsilon(t, x).
\end{equation}
Conversely, let $S \in \mathcal{B}_m([0, T] \times H)$ be a solution to the shifted formulation \eqref{eq:shifted_formulation}. Let $\tilde{S}$ be the mapping obtained by substituting $S$ into the direct mild formulation:
\begin{equation}
\label{eq:proof_S_tilde}
    \tilde{S}(t, x) := P_{t,T}[\Psi](x) + \int_t^T P_{t,s} \left[ f(s, \cdot) - \frac{1}{\varepsilon}(S(s, \cdot) - \Phi(s, \cdot))^+ \right](x) \, \mathrm{d}s.
\end{equation}
Defining $H_S(s, \cdot) := f(s, \cdot) - \frac{1}{\varepsilon}(S(s, \cdot) - \Phi(s, \cdot))^+$, the direct implication above proves that $\tilde{S}$ satisfies the shifted formulation. Subtracting the shifted equations for $S$ and $\tilde{S}$:
\begin{equation}
\label{eq:proof_difference}
    S(t, x) - \tilde{S}(t, x) = \int_t^T L e^{-L(s-t)} P_{t,s} \big[ S(s, \cdot) - \tilde{S}(s, \cdot) \big](x) \, \mathrm{d}s.
\end{equation}
Using the stability estimate $P_{t,s}[w_m](x) \le K_m w_m(x)$ from Lemma \ref{lem:moment_estimates} and bounding $e^{-L(s-t)} \le 1$, we deduce:
\begin{equation}
\label{eq:proof_gronwall_ineq}
    \|S(t, \cdot) - \tilde{S}(t, \cdot)\|_{\mathcal{B}_m} \le L K_m \int_t^T \|S(s, \cdot) - \tilde{S}(s, \cdot)\|_{\mathcal{B}_m} \, \mathrm{d}s.
\end{equation}
Applying the backward Gronwall's inequality implies that $\|S(t, \cdot) - \tilde{S}(t, \cdot)\|_{\mathcal{B}_m} = 0$ for all $t \in [0, T]$. Thus $S \equiv \tilde{S}$, confirming that $S$ is the unique mild solution.
\end{proof}

To prove existence and uniqueness of mild solutions, we introduce the following nonlinear integral operator.

\begin{definition}[Shifted Integral Operator]
\label{def:shifted_operator}
Let $\varepsilon > 0$ and consider a shift parameter $L \ge 0$. The \emph{nonlinear integral operator} $\mathcal{M}_{\varepsilon, L} : \mathcal{B}_m([0, T] \times H) \to \mathcal{B}_m([0, T] \times H)$ is defined by:
\begin{equation} 
\label{eq:shifted_operator_eq}
    \mathcal{M}_{\varepsilon, L}(u)(t, x) := e^{-L(T-t)} P_{t,T}[\Psi](x) + \int_t^T e^{-L(s-t)} P_{t,s} \left[ f(s, \cdot) + L u(s, \cdot) - \frac{1}{\varepsilon} \left( u(s, \cdot) - \Phi(s, \cdot) \right)^+ \right](x) \, \mathrm{d}s.
\end{equation}
\end{definition}

\begin{remark}[Measurability and Topological Invariance] 
\label{rem:topological_invariance}
The integral operator $\mathcal{M}_{\varepsilon, L}$ maps $\mathcal{B}_m([0, T] \times H)$ into itself. Borel measurability is preserved by the action of the transition semigroup and the temporal integration. Moreover, the uniform parabolic moment estimates (Lemma \ref{lem:moment_estimates}) ensure that the transition semigroup preserves the polynomial growth envelope. This guarantees a finite $\|\cdot\|_{\mathcal{B}_{m, T}}$ norm.
\end{remark}

We now prove the existence and uniqueness of the penalized mild solution $u_\varepsilon$ within this measurable framework.

\begin{theorem}[Existence and Uniqueness] 
\label{thm:existence_uniqueness}
Let Assumptions \ref{ass:coefficients} and \ref{ass:problem_data} hold. For any $\varepsilon > 0$ and any shift parameter $L \ge 1/\varepsilon$, the nonlinear operator $\mathcal{M}_{\varepsilon, L}$ admits a unique fixed point $u_\varepsilon \in \mathcal{B}_m([0, T] \times H)$, which corresponds to the unique mild solution of the penalized equation \eqref{eq:hjb_penalized}.
\end{theorem}

\begin{proof}
Let $u, v \in \mathcal{B}_m([0, T] \times H)$. To establish the contraction property of the operator $\mathcal{M}_{\varepsilon, L}$, we evaluate the pointwise difference between the operator applied to $u$ and $v$. Subtracting the respective representations, the terminal terms cancel out. By the linearity and positivity of the transition semigroup, we obtain:
\begin{gather}
\label{eq:operator_diff}
    |\mathcal{M}_{\varepsilon, L}(u)(t, x) - \mathcal{M}_{\varepsilon, L}(v)(t, x)| \le\\ \int_t^T e^{-L(s-t)} P_{t,s} \left[ \Big| L u(s, \cdot) - \frac{1}{\varepsilon}(u(s, \cdot) - \Phi)^+ - \Big( L v(s, \cdot) - \frac{1}{\varepsilon}(v(s, \cdot) - \Phi)^+ \Big) \Big| \right](x) \, \mathrm{d}s.
\end{gather}
To bound the integrand in the previous step, we observe that for any choice of $L \ge 1/\varepsilon$ and any constant $c \in \R$, the real-valued mapping $r \mapsto L r - \frac{1}{\varepsilon}(r - c)^+$ is globally $L$-Lipschitz continuous. Applying this property pointwise with $c = \Phi(s, y)$ yields:
\begin{equation}
\label{eq:lipschitz_bound}
    \Big| L u(s, y) - \frac{1}{\varepsilon}(u(s, y) - \Phi(s, y))^+ - \Big( L v(s, y) - \frac{1}{\varepsilon}(v(s, y) - \Phi(s, y))^+ \Big) \Big| \le L |u(s, y) - v(s, y)|.
\end{equation}
Extracting the weighted norm associated with the measurable parabolic space $\mathcal{B}_m([0, T] \times H)$:
\begin{equation}
\label{eq:weighted_norm_bound}
    |u(s, y) - v(s, y)| \le \|u(s, \cdot) - v(s, \cdot)\|_{\mathcal{B}_m} w_m(y) \le \|u - v\|_{\mathcal{B}_{m, T}} w_m(y).
\end{equation}
Substituting the bound \eqref{eq:weighted_norm_bound} back into the integral \eqref{eq:operator_diff}, and leveraging the exact transition operator growth estimate $P_{t,s}[w_m](x) \le K_m w_m(x)$ from Lemma \ref{lem:moment_estimates}, we get:
\begin{equation}
\label{eq:integral_bound_1}
    |\mathcal{M}_{\varepsilon, L}(u)(t, x) - \mathcal{M}_{\varepsilon, L}(v)(t, x)| \le L K_m \|u - v\|_{\mathcal{B}_{m, T}} w_m(x) \int_t^T e^{-L(s-t)} \, \mathrm{d}s.
\end{equation}
Bounding the exponential integrand by $e^{-L(s-t)} \le 1$, we obtain the exact base case for the iteration:
\begin{equation}
\label{eq:base_case}
    |\mathcal{M}_{\varepsilon, L}(u)(t, x) - \mathcal{M}_{\varepsilon, L}(v)(t, x)| \le L K_m \|u - v\|_{\mathcal{B}_{m, T}} w_m(x) (T-t).
\end{equation}
Proceeding by standard induction on the operator composition, evaluating the repeated temporal integrals of Volterra type yields the estimate for the $n$-th iterate:
\begin{equation}
\label{eq:n_iterate}
    |\mathcal{M}^n_{\varepsilon, L}(u)(t, x) - \mathcal{M}^n_{\varepsilon, L}(v)(t, x)| \le \frac{(L K_m)^n (T-t)^n}{n!} \|u - v\|_{\mathcal{B}_{m, T}} w_m(x).
\end{equation}
We divide by the spatial weight $w_m(x)$ and take the supremum over the domain $[0, T] \times H$, obtaining the global norm bound:
\begin{equation}
\label{eq:global_n_iterate}
    \|\mathcal{M}^n_{\varepsilon, L}(u) - \mathcal{M}^n_{\varepsilon, L}(v)\|_{\mathcal{B}_{m, T}} \le \frac{(L K_m T)^n}{n!} \|u - v\|_{\mathcal{B}_{m, T}}.
\end{equation}
 Consequently, there exists an integer $n \ge 1$ such that $\mathcal{M}^n_{\varepsilon, L}$ becomes a contraction on $\mathcal{B}_m([0, T] \times H)$. By the generalized Banach Fixed-Point Theorem, the operator $\mathcal{M}_{\varepsilon, L}$ admits a unique fixed point $u_\varepsilon$, concluding the proof.
\end{proof}

We now prove the joint continuity of the solution to the penalized backward equation \eqref{eq:hjb_penalized}.

\begin{corollary}[Joint Continuity of the Penalized Solution]
\label{cor:continuity_penalized}
Under Assumptions \ref{ass:coefficients} and \ref{ass:problem_data}, the unique mild solution $u_\varepsilon$ of the penalized equation \eqref{eq:hjb_penalized} belongs to $\mathcal{C}_m([0, T] \times H)$.
\end{corollary}

\begin{proof}
We must show that the contraction operator $\mathcal{M}_{\varepsilon, L}$ maps $\mathcal{C}_m([0, T] \times H)$ into itself. Let $u \in \mathcal{C}_m([0, T] \times H)$. By the hypotheses on the problem data, the terminal mapping $\Psi$, the forcing term $f$, and the obstacle $\Phi$ are continuous. Therefore, the source term in the integral representation, defined as:
\begin{equation*}
    g_u(s, x) := f(s, x) + L u(s, x) - \frac{1}{\varepsilon}(u(s, x) - \Phi(s, x))^+,
\end{equation*}
is a mapping belonging to $\mathcal{C}_m([0, T] \times H)$. The action of the operator is given by:
\begin{equation}
\label{eq:operator_action_continuity}
    \mathcal{M}_{\varepsilon, L}(u)(t, x) = e^{-L(T-t)} P_{t,T}[\Psi](x) + \int_t^T e^{-L(s-t)} P_{t,s}[g_u(s, \cdot)](x) \, \mathrm{d}s.
\end{equation}
We analyze the terms on the right-hand side of \eqref{eq:operator_action_continuity} to establish joint continuity at an arbitrary point $(t_0, x_0) \in [0, T] \times H$. Let $(t_n, x_n) \to (t_0, x_0)$. For the terminal condition, the mean-square continuity of the mild stochastic flow $(t,x) \mapsto X(T; t, x)$ with respect to the initial data, combined with the spatial continuity and polynomial growth of $\Psi$, ensures that $(t, x) \mapsto P_{t,T}[\Psi](x) = \mathbb{E}[\Psi(X(T; t, x))]$ is jointly continuous. For the integral term, we decompose the difference via the triangle inequality:
\begin{align*}
    &\left| \int_{t_n}^T e^{-L(s-t_n)} P_{t_n,s}[g_u(s, \cdot)](x_n) \, \mathrm{d}s - \int_{t_0}^T e^{-L(s-t_0)} P_{t_0,s}[g_u(s, \cdot)](x_0) \, \mathrm{d}s \right| \\
    &\quad \le \int_{t_0 \lor t_n}^T \left| e^{-L(s-t_n)} P_{t_n,s}[g_u(s, \cdot)](x_n) - e^{-L(s-t_0)} P_{t_0,s}[g_u(s, \cdot)](x_0) \right| \mathrm{d}s \\
    & \quad \quad + \left| \int_{t_n}^{t_0 \lor t_n} e^{-L(s-t_n)} P_{t_n,s}[g_u(s, \cdot)](x_n) \, \mathrm{d}s \right| + \left| \int_{t_0}^{t_0 \lor t_n} e^{-L(s-t_0)} P_{t_0,s}[g_u(s, \cdot)](x_0) \, \mathrm{d}s \right|.
\end{align*}
The last two terms represent integrals over regions whose measure $|t_n - t_0|$ shrinks to zero as $n \to \infty$. Since the convergent sequence $(x_n)_{n \in \mathbb{N}}$ is bounded in $H$, the stability estimate $P_{t,s}[w_m](x_n) \le K_m w_m(x_n)$ from Lemma \ref{lem:moment_estimates} ensures the integrands are uniformly bounded by a constant independent of $n$, hence these terms vanish. For the first term, the integrand converges to zero pointwise almost everywhere as $n \to \infty$ due to the joint continuity of $(t,x) \mapsto \mathbb{E}[g_u(s, X(s; t, x))]$ for $s > t_0 \lor t_n$. Relying on the same uniform bound established above, the Dominated Convergence Theorem ensures that this integral also vanishes. Consequently, $\mathcal{M}_{\varepsilon, L}(u)$ is jointly continuous, confirming that $\mathcal{M}_{\varepsilon, L}$ maps the closed subspace $\mathcal{C}_m([0, T] \times H)$ into itself. Since $\mathcal{M}_{\varepsilon, L}$ remains a contraction on this subspace (via its $n$-th iterate), it admits a unique fixed point within $\mathcal{C}_m([0, T] \times H)$. By the uniqueness established in Theorem \ref{thm:existence_uniqueness}, this continuous fixed point must coincide with $u_\varepsilon$.
\end{proof}

The following result ensures the stability of the mild solutions of \ref{eq:hjb_penalized} with respect to variations of the data.

\begin{lemma}[Stability of the Mild Solution] 
\label{lem:stability_mild}
Let Assumptions \ref{ass:coefficients} and \ref{ass:problem_data} hold. For a fixed penalization parameter $\varepsilon > 0$, let $u_{\varepsilon, 1}$ and $u_{\varepsilon, 2}$ be the respective mild solutions of \eqref{eq:hjb_penalized} driven by the problem data triplets $(f_1, \Phi_1, \Psi_1)$ and $(f_2, \Phi_2, \Psi_2)$. Then there exists a constant $C_T > 0$ such that:
\begin{equation*}
    \|u_{\varepsilon, 1}(t, \cdot) - u_{\varepsilon, 2}(t, \cdot)\|_{\mathcal{B}_m} \le C_T \left( \|\Psi_1 - \Psi_2\|_{\mathcal{B}_m} + \|f_1 - f_2\|_{\mathcal{B}_{m, T}} + \frac{1}{\varepsilon} \|\Phi_1 - \Phi_2\|_{\mathcal{B}_{m, T}} \right).
\end{equation*}
\end{lemma}

\begin{proof}
To estimate the difference between the mild solutions $u_{\varepsilon, 1}$ and $u_{\varepsilon, 2}$, we use the shifted integral formulation established in Theorem \ref{thm:equivalence_mild}, setting the shift parameter $L := 1/\varepsilon$. Using the Lipschitz continuity of the non-linear operator $r \mapsto Lr - \frac{1}{\varepsilon}(r - c)^+$, the linearity of the transition evolution family, and the triangle inequality, we obtain:
\begin{align*} 
\label{eq:stability_initial_diff}
    |u_{\varepsilon, 1}(t, x) - u_{\varepsilon, 2}(t, x)| &\le e^{-L(T-t)} P_{t,T}[|\Psi_1 - \Psi_2|](x) \nonumber \\
    & \quad + \int_t^T e^{-L(s-t)} P_{t,s}\left[ |f_1(s, \cdot) - f_2(s, \cdot)| + L |u_{\varepsilon, 1}(s, \cdot) - u_{\varepsilon, 2}(s, \cdot)| + \frac{1}{\varepsilon} |\Phi_1(s, \cdot) - \Phi_2(s, \cdot)| \right](x) \, \mathrm{d}s.
\end{align*}
Dividing by the spatial weight $w_m(x)$ and applying the stability estimate $P_{t,s}[w_m](x) \le K_m w_m(x)$ from Lemma \ref{lem:moment_estimates}, we obtain:
\begin{align*}
    \|u_{\varepsilon, 1}(t, \cdot) - u_{\varepsilon, 2}(t, \cdot)\|_{\mathcal{B}_m} &\le K_m \|\Psi_1 - \Psi_2\|_{\mathcal{B}_m} \\
    & \quad + K_m \int_t^T e^{-L(s-t)} \left( \|f_1 - f_2\|_{\mathcal{B}_{m, T}} + \frac{1}{\varepsilon} \|u_{\varepsilon, 1} - u_{\varepsilon, 2}\|_{\mathcal{B}_m} + \frac{1}{\varepsilon} \|\Phi_1 - \Phi_2\|_{\mathcal{B}_{m, T}} \right) \mathrm{d}s.
\end{align*}
Since $e^{-L(s-t)} \le 1$ for $s \ge t$ and $L=1/\varepsilon$, we deduce:
\begin{align*}
    \|u_{\varepsilon, 1}(t, \cdot) - u_{\varepsilon, 2}(t, \cdot)\|_{\mathcal{B}_m} &\le K_m \left( \|\Psi_1 - \Psi_2\|_{\mathcal{B}_m} + T \|f_1 - f_2\|_{\mathcal{B}_{m, T}} + \frac{T}{\varepsilon} \|\Phi_1 - \Phi_2\|_{\mathcal{B}_{m, T}} \right) \\
    & \quad + \frac{K_m}{\varepsilon} \int_t^T \|u_{\varepsilon, 1}(s, \cdot) - u_{\varepsilon, 2}(s, \cdot)\|_{\mathcal{B}_m} \, \mathrm{d}s.
\end{align*}
Applying backward Gronwall's Lemma yields:
\begin{equation}
\label{eq:stability_gronwall_result}
    \|u_{\varepsilon, 1}(t, \cdot) - u_{\varepsilon, 2}(t, \cdot)\|_{\mathcal{B}_m} \le K_m \left[ \|\Psi_1 - \Psi_2\|_{\mathcal{B}_m} + T \|f_1 - f_2\|_{\mathcal{B}_{m, T}} + \frac{T}{\varepsilon} \|\Phi_1 - \Phi_2\|_{\mathcal{B}_{m, T}} \right] e^{\frac{K_m (T-t)}{\varepsilon}}.
\end{equation}
By defining $C_T := K_m \max(1, T, \frac{T}{\varepsilon}) e^{\frac{K_m T}{\varepsilon}}$, the estimate simplifies to the desired stability bound, concluding the proof.
\end{proof}

\subsection{The Penalized Minimization Problem}
\label{sec:randomized}
Following the line of \cite[Chapter 3, Section 4, Theorem 4.4]{BeLi82}, we introduce the probabilistic framework for the penalized minimization problem. 

\begin{definition}[Control and Cost]
\label{def:randomized_control}
\begin{enumerate}[label={(\textit{\roman*})}, leftmargin=*]
    \item[]
    \item  \emph{Admissible Control Space.} Let $\mathcal{V}_\varepsilon$ denote the space of all $\F$-progressively measurable processes $\gamma: \Omega \times [0, T] \to \R$ taking values $\mathbb{P}$-a.s. in the closed interval $[0, 1/\varepsilon]$.
    
    \item  \emph{Penalized Cost Functional.} For any control process $\gamma \in \mathcal{V}_\varepsilon$, the associated penalized cost functional $J_\varepsilon: [0, T] \times H \times \mathcal{V}_\varepsilon \to \R$ is defined by:
    \begin{equation} 
    \label{eq:randomized_cost}
        J_\varepsilon(t, x; \gamma) := \E \left[ \int_t^T e^{-\int_t^s \gamma_r \mathrm{d}r} \Big( f(s, X(s)) + \gamma_s \Phi(s, X(s)) \Big) \mathrm{d}s + e^{-\int_t^T \gamma_r \mathrm{d}r} \Psi(X(T)) \right],
    \end{equation}
    where $X(s) \equiv X(s; t, x)$ is the continuous mild solution \eqref{eq:mild_solution_eq} of the underlying semilinear state equation \eqref{eq:cauchy_problem}.
\end{enumerate}
\end{definition}

The following theorem establishes a probabilistic representation, linking the penalized backward equation to the penalized minimization problem via the martingale property of the mild solution.

\begin{theorem}[Probabilistic Representation]
\label{thm:feynman_kac}
Let Assumptions \ref{ass:coefficients} and \ref{ass:problem_data} hold. The mild solution $u_\varepsilon$ to the penalized backward equation \eqref{eq:hjb_penalized} coincides with the minimal penalized cost functional \eqref{eq:randomized_cost} over the admissible control space $\mathcal{V}_\varepsilon$:
\begin{equation} 
    u_\varepsilon(t, x) = \inf_{\gamma \in \mathcal{V}_\varepsilon} J_\varepsilon(t, x; \gamma).
\end{equation}
\end{theorem}

\begin{proof}
Evaluating the mild solution of equation \eqref{eq:hjb_penalized} along the stochastic flow and applying the Markov property yields, for any $s \in [t, T]$:
\begin{equation*}
    u_\varepsilon(s, X(s)) = \E \left[ \Psi(X(T)) + \int_s^T \left( f(r, X(r)) - \frac{1}{\varepsilon}(u_\varepsilon(r, X(r)) - \Phi(r, X(r)))^+ \right) \mathrm{d}r \;\middle|\; \mathcal{F}_s \right].
\end{equation*}
Adding the accumulated running cost up to time $s$, we introduce the compensated process:
\begin{equation*}
    N^\varepsilon_s := u_\varepsilon(s, X(s)) + \int_t^s \left( f(r, X(r)) - \frac{1}{\varepsilon}(u_\varepsilon(r, X(r)) - \Phi(r, X(r)))^+ \right) \mathrm{d}r.
\end{equation*}
Assumption \ref{ass:problem_data} and Lemma \ref{lem:moment_estimates} guarantee the integrability $\E[\sup_{s \in [t, T]} |N^\varepsilon_s|] < \infty$. Since $\E[N^\varepsilon_T \mid \mathcal{F}_s] = N^\varepsilon_s$, the process is a continuous $\F$-martingale. Substituting the terminal condition $u_\varepsilon(T, X(T)) = \Psi(X(T))$ gives:
\begin{equation} 
\label{eq:cond_exp_u}
    u_\varepsilon(s, X(s)) = \E \left[ \Psi(X(T)) + \int_s^T \left( f(r, X(r)) - \frac{1}{\varepsilon}(u_\varepsilon(r, X(r)) - \Phi(r, X(r)))^+ \right) \mathrm{d}r \;\middle|\; \mathcal{F}_s \right].
\end{equation}
To connect this representation to the penalized minimization problem, we fix an admissible control $\gamma \in \mathcal{V}_\varepsilon$ and define the integral quantity:
\begin{equation*}
    I := \E \left[ \int_t^T \gamma_s e^{-\int_t^s \gamma_q \mathrm{d}q} u_\varepsilon(s, X(s)) \, \mathrm{d}s \right].
\end{equation*}
Substituting the expression for $u_\varepsilon(s, X(s))$ from \eqref{eq:cond_exp_u} into $I$ and applying the tower property of conditional expectations yields:
\begin{equation*}
    I = \E \left[ \int_t^T \gamma_s e^{-\int_t^s \gamma_q \mathrm{d}q} \left( \Psi(X(T)) + \int_s^T \left( f(r, X(r)) - \frac{1}{\varepsilon}(u_\varepsilon(r, X(r)) - \Phi(r, X(r)))^+ \right) \mathrm{d}r \right) \mathrm{d}s \right].
\end{equation*}
Applying Fubini's Theorem and computing the inner exponential integral gives:
\begin{align*}
    &\E \left[ \int_t^T \left( \int_s^T \gamma_s e^{-\int_t^s \gamma_q \mathrm{d}q} \left( f(r, X(r)) - \frac{1}{\varepsilon}(u_\varepsilon(r, X(r)) - \Phi(r, X(r)))^+ \right) \mathrm{d}r \right) \mathrm{d}s \right] \\
    &\quad = \E \left[ \int_t^T \left( f(r, X(r)) - \frac{1}{\varepsilon}(u_\varepsilon(r, X(r)) - \Phi(r, X(r)))^+ \right) \left( \int_t^r \gamma_s e^{-\int_t^s \gamma_q \mathrm{d}q} \, \mathrm{d}s \right) \mathrm{d}r \right],
\end{align*}
where, clearly, $\int_t^r \gamma_s e^{-\int_t^s \gamma_q \mathrm{d}q} \, \mathrm{d}s=1 - e^{-\int_t^r \gamma_q \mathrm{d}q}$. Evaluating the martingale $N^\varepsilon$ at time $t$ provides the identity $u_\varepsilon(t, x) = \E[N^\varepsilon_T]$. Expanding this expectation gives:
\begin{equation} 
\label{eq:undiscounted_u}
    u_\varepsilon(t, x) = \E \left[ \Psi(X(T)) + \int_t^T \left( f(s, X(s)) - \frac{1}{\varepsilon}(u_\varepsilon(s, X(s)) - \Phi(s, X(s)))^+ \right) \mathrm{d}s \right].
\end{equation}
Subtracting the integral $I$ from \eqref{eq:undiscounted_u} recovers the exponential discounting associated with the penalized minimization problem, yielding:
\begin{align*}
    u_\varepsilon(t, x) - I &= \E \Bigg[ e^{-\int_t^T \gamma_q \mathrm{d}q} \Psi(X(T))+ \int_t^T e^{-\int_t^s \gamma_q \mathrm{d}q} \left( f(s, X(s)) - \frac{1}{\varepsilon}(u_\varepsilon(s, X(s)) - \Phi(s, X(s)))^+ \right) \mathrm{d}s \Bigg].
\end{align*}
Exploiting the linearity of the expectation the following representation for $u_\varepsilon(t, x)$ is obtained:
\begin{equation}\label{eq:rep_uepsilon}
 \E \Bigg[ e^{-\int_t^T \gamma_q \mathrm{d}q} \Psi(X(T)) + \int_t^T e^{-\int_t^s \gamma_q \mathrm{d}q} \Big( f(s, X(s)) - \frac{1}{\varepsilon}(u_\varepsilon(s, X(s)) - \Phi(s, X(s)))^+ + \gamma_s u_\varepsilon(s, X(s)) \Big) \mathrm{d}s \Bigg].
\end{equation}
We observe that the admissible control satisfies $\gamma_s \in [0, 1/\varepsilon]$ $\Prob$-a.s., ensuring the following inequality:
\begin{equation*}
    \gamma_s u_\varepsilon - \frac{1}{\varepsilon}(u_\varepsilon - \Phi)^+ \le \gamma_s u_\varepsilon - \gamma_s(u_\varepsilon - \Phi) =\gamma_s \Phi.
\end{equation*}
Applying this bound to the representation of $u_\varepsilon(t, x)$ obtained in \eqref{eq:rep_uepsilon} yields the penalized cost functional introduced in \eqref{eq:randomized_cost}:
\begin{equation*}
    u_\varepsilon(t, x) \le \E \left[ \int_t^T e^{-\int_t^s \gamma_q \mathrm{d}q} \Big( f(s, X(s)) + \gamma_s \Phi(s, X(s)) \Big) \mathrm{d}s + e^{-\int_t^T \gamma_q \mathrm{d}q} \Psi(X(T)) \right] \equiv J_\varepsilon(t, x; \gamma).
\end{equation*}
Finally, to prove that this lower bound is attained, we define the specific measurable control process $\gamma^* \in \mathcal{V}_\varepsilon$:
\begin{equation*}
    \gamma^*_s := \frac{1}{\varepsilon} \ind_{\{u_\varepsilon(s, X(s)) > \Phi(s, X(s))\}}.
\end{equation*}
Evaluating the expression with $\gamma^*$ yields the pointwise identity $\gamma^*_s u_\varepsilon - \frac{1}{\varepsilon}(u_\varepsilon - \Phi)^+ = \gamma^*_s \Phi$, $\mathbb{P}$-a.s. Consequently, the lower bound is achieved, proving $u_\varepsilon(t, x) = J_\varepsilon(t, x; \gamma^*)$.
\end{proof}

\subsection{Weighted Equi-Lipschitz Continuity}
\label{sec:equi_lipschitz}

We establish the uniform equi-Lipschitz continuity of the penalized family $\{u_\varepsilon\}_{\varepsilon > 0}$ in the weighted polynomial topology. This property is derived exclusively via probabilistic arguments, without requiring any smoothing assumptions on the transition semigroup.

We start with the following useful Lemma.
\begin{lemma}[Mixed Moment Bound]
\label{lem:mixed_moment}
Let Assumption \ref{ass:coefficients} hold. Define the \emph{spatial variation weight} $W: H \times H \to \R$ by:
\begin{equation}
\label{eq:spatial_weight}
    W(x, y) := (1 + \|x\|_H^{m-1} + \|y\|_H^{m-1})\|x - y\|_H.
\end{equation}
Then there exists a constant $M_2 > 0$, depending only on $T$, $m$ and the Lipschitz constants, such that for all $x, y \in H$ and $t \in [0, T]$, the mild solutions \eqref{eq:mild_solution_eq} satisfy:
\begin{equation}
\label{eq:mixed_moment}
    \E \left[ \sup_{s \in [t, T]} W\big( X(s; t, x), X(s; t, y) \big) \right] \le M_2 W(x, y).
\end{equation}
\end{lemma}

\begin{proof}
    The proof is a consequence of the Lipschitz continuity of the coefficients (Assumption \ref{ass:coefficients}), Hölder's inequality, Burkholder-Davis-Gundy inequality and Gronwall's Lemma.
\end{proof}

Now, we exploit the penalized representation to derive uniform spatial variation bounds for the penalized sequence.

\begin{theorem}[Uniform Lipschitz Envelope] 
\label{thm:uniform_lipschitz}
Let Assumptions \ref{ass:coefficients} and \ref{ass:problem_data} hold. For any penalization parameter $\varepsilon > 0$, the parabolic mild solution $u_\varepsilon$ to the penalized backward equation \eqref{eq:hjb_penalized} is locally Lipschitz continuous in space. Moreover, there exists a continuous envelope $K(t) > 0$, independent of $\varepsilon$, such that for all $x, y \in H$ and $t \in [0, T]$:
\begin{equation} 
\label{eq:uniform_lipschitz}
    |u_\varepsilon(t, x) - u_\varepsilon(t, y)| \le K(t) W(x, y),
\end{equation}
where $W(x, y)$ is the spatial variation weight defined in \eqref{eq:spatial_weight}.
\end{theorem}

\begin{proof}
By Theorem \ref{thm:feynman_kac}, the penalized solution is the value function of the penalized minimization problem:
\begin{equation*}
    u_\varepsilon(t, x) = \inf_{\gamma \in \mathcal{V}_\varepsilon} J_\varepsilon(t, x; \gamma).
\end{equation*}
The non-expansive property of the infimum allows us to bound the spatial difference by the supremum of the cost functional differences:
\begin{equation}
\label{eq:proof_inf_nonexpansive}
    |u_\varepsilon(t, x) - u_\varepsilon(t, y)| \le \sup_{\gamma \in \mathcal{V}_\varepsilon} |J_\varepsilon(t, x; \gamma) - J_\varepsilon(t, y; \gamma)|.
\end{equation}
Let $A_s := \int_t^s \gamma_r \, \mathrm{d}r$ be the cumulative control process. By the linearity of the expectation and the triangle inequality, the difference between cost functionals satisfies:
\begin{align*}
\label{eq:proof_cost_diff_bound}
    |J_\varepsilon(t, x; \gamma) - J_\varepsilon(t, y; \gamma)| &\le \E \Bigg[ \int_t^T e^{-A_s} |f(s, X(s; t, x)) - f(s, X(s; t, y))| \, \mathrm{d}s \nonumber \\
    & \quad + \int_t^T \gamma_s e^{-A_s} |\Phi(s, X(s; t, x)) - \Phi(s, X(s; t, y))| \, \mathrm{d}s + e^{-A_T} |\Psi(X(T; t, x)) - \Psi(X(T; t, y))| \Bigg].
\end{align*}
Under Assumption \ref{ass:problem_data}, $f, \Phi$, and $\Psi$ are Lipschitz continuous in space with a uniform constant $L$. Defining $W^* := \sup_{s \in [t, T]} W(X(s; t, x), X(s; t, y))$ and factoring the Lipschitz constants yields:
\begin{equation}
\label{eq:proof_factored_cost_diff}
    |J_\varepsilon(t, x; \gamma) - J_\varepsilon(t, y; \gamma)| \le \E \left[ W^* L\left( \int_t^T e^{-A_s} \, \mathrm{d}s + \int_t^T \gamma_s e^{-A_s} \, \mathrm{d}s + e^{-A_T} \right) \right].
\end{equation}
Since $\int_t^T e^{-A_s} \, \mathrm{d}s \le (T-t)$ and $\int_t^T \gamma_s e^{-A_s} \, \mathrm{d}s = 1 - e^{-A_T}$, the terminal term $e^{-A_T}$ cancels exactly, providing:
\begin{equation}
\label{eq:proof_pre_mixed_moment}
    |J_\varepsilon(t, x; \gamma) - J_\varepsilon(t, y; \gamma)| \le L(T-t+1) \E [W^*].
\end{equation}
Applying the mixed moment bound $\E[W^*] \le M_2 W(x, y)$ from Lemma \ref{lem:mixed_moment} and defining the temporal envelope $K(t) := M_2 L (T-t+1)$, we obtain:
\begin{equation*}
    |J_\varepsilon(t, x; \gamma) - J_\varepsilon(t, y; \gamma)| \le K(t) W(x, y).
\end{equation*}
The bound holds uniformly for all $\gamma \in \mathcal{V}_\varepsilon$, establishing the Lipschitz continuity of $u_\varepsilon$ with respect to the weight $W$.
\end{proof}
\section{Convergence of the Penalized Solutions and Verification}
\noindent In this section we investigate the asymptotic convergence of the family of penalized solutions as the regularization parameter $\varepsilon$ vanishes. In particular, we prove that the limit of the solutions to the penalized equations \eqref{eq:hjb_penalized} coincide with the unique mild solution to the variational inequality. Moreover, we show that we can identify this solution with the value function of the optimal stopping problem.

\subsection{The Lower Bound}
We derive a uniform lower bound for the expected cost, establishing a time dependent bound that is independent of both the control $\gamma$ and the penalization parameter $\varepsilon$.

To establish this result, we first introduce a deterministic dynamic barrier governed by the problem's data.
\begin{definition}[Dynamic Barrier]
\label{def:dynamic_barrier}
Let $K_m \ge 1$ denote the uniform constant associated with the parabolic maximal inequality \eqref{eq:moment_estimates_eq} on $H$. The \emph{dynamic bounding barrier} $Z_{low}: [0, T] \times H \to \R$ is defined by:
\begin{equation} 
\label{eq:z_low}
    Z_{low}(t, x) := - K_m \Big( \max\left(\|\Phi\|_{\mathcal{B}_{m, T}}, \|\Psi\|_{\mathcal{B}_m}\right) + (T-t)\|f\|_{\mathcal{B}_{m, T}} \Big) w_m(x).
\end{equation}
\end{definition}

Next, we derive a pathwise integral estimate that controls the cost components independently of the chosen control.
\begin{lemma}[Pathwise Bound]
\label{lem:pathwise_bound}
Let Assumption \ref{ass:problem_data} hold. Define the cumulative control process $A_s := \int_t^s \gamma_r \, \mathrm{d}r$. For any progressively measurable control $\gamma \in \mathcal{V}_\varepsilon$ (Definition \ref{def:randomized_control}), the following pathwise integral bound holds $\mathbb{P}$-a.s.:
\begin{equation*}
    \int_t^T e^{-A_s} \|f\|_{\mathcal{B}_{m, T}} \, \mathrm{d}s + \int_t^T \gamma_s e^{-A_s} \|\Phi\|_{\mathcal{B}_{m, T}} \, \mathrm{d}s + e^{-A_T} \|\Psi\|_{\mathcal{B}_m} \le (T-t)\|f\|_{\mathcal{B}_{m, T}} + \max\left(\|\Phi\|_{\mathcal{B}_{m, T}}, \|\Psi\|_{\mathcal{B}_m}\right).
\end{equation*}
\end{lemma}

\begin{proof}
Since $\gamma_s$ is the derivative of the cumulative process $A_s := \int_t^s \gamma_r \, \mathrm{d}r$, direct integration over $[t, T]$ yields:
\begin{equation*}
    \int_t^T \gamma_s e^{-A_s} \|\Phi\|_{\mathcal{B}_{m, T}} \, \mathrm{d}s = \|\Phi\|_{\mathcal{B}_{m, T}} \left( 1 - e^{-A_T} \right).
\end{equation*}
Since the admissible control is non-negative $\mathbb{P}$-a.s., $A_s \ge 0$ and $e^{-A_s} \le 1$. Applying this pointwise bound to the running cost provides:
\begin{equation*}
    \int_t^T e^{-A_s} \|f\|_{\mathcal{B}_{m, T}} \, \mathrm{d}s \le (T-t)\|f\|_{\mathcal{B}_{m, T}}.
\end{equation*}
Defining the maximal envelope $M_\Phi^\Psi := \max\left(\|\Phi\|_{\mathcal{B}_{m, T}}, \|\Psi\|_{\mathcal{B}_m}\right)$, we note that $e^{-A_T} \in [0, 1]$ $\mathbb{P}$-a.s. Consequently, the obstacle and terminal components form a convex combination strictly bounded by $M_\Phi^\Psi$:
\begin{equation*}
    \|\Phi\|_{\mathcal{B}_{m, T}} \left( 1 - e^{-A_T} \right) + \|\Psi\|_{\mathcal{B}_m} e^{-A_T} \le M_\Phi^\Psi.
\end{equation*}
Summing this estimate with the running cost bound establishes the desired pathwise inequality.
\end{proof}

Exploiting this pathwise estimate, we now prove that the family of penalized solutions is uniformly bounded from below by the dynamic barrier.
\begin{theorem}[Uniform Lower Boundedness]
\label{thm:uniform_lower_bound}
Let Assumptions \ref{ass:coefficients} and \ref{ass:problem_data} hold. The mild solution $u_\varepsilon$ to the penalized backward equation \eqref{eq:hjb_penalized} is uniformly bounded from below by the dynamic barrier $Z_{low}$ defined in \eqref{eq:z_low}:
\begin{equation} 
\label{eq:uniform_lower_bound}
    u_\varepsilon(t, x) \ge Z_{low}(t, x), \quad \forall (t, x) \in [0, T] \times H.
\end{equation}
\end{theorem}

\begin{proof}
For $g \in \{f, \Phi\}$, $g(s, y) \ge -\|g\|_{\mathcal{B}_{m, T}} w_m(y)$ and $\Psi(y) \ge -\|\Psi\|_{\mathcal{B}_m} w_m(y)$. Since $\gamma_s \ge 0$, substituting into the penalized cost representation yields:
\begin{equation*}
    J_\varepsilon(t, x; \gamma) \ge - \E \Bigg[ \int_t^T e^{-A_s} \Big( \|f\|_{\mathcal{B}_{m, T}} + \gamma_s \|\Phi\|_{\mathcal{B}_{m, T}} \Big) w_m(X(s;t,x)) \, \mathrm{d}s + e^{-A_T} \|\Psi\|_{\mathcal{B}_m} w_m(X(T;t,x)) \Bigg].
\end{equation*}
Let $W^* := \sup_{r \in [t, T]} w_m(X(r;t,x))$. Bounding $w_m(X(s;t,x))$ and $w_m(X(T;t,x))$ by $W^*$ and applying Lemma \ref{lem:pathwise_bound} gives:
\begin{align*}
    J_\varepsilon(t, x; \gamma) &\ge - \E \left[ W^* \left( \int_t^T e^{-A_s} \|f\|_{\mathcal{B}_{m, T}} \, \mathrm{d}s + \int_t^T \gamma_s e^{-A_s} \|\Phi\|_{\mathcal{B}_{m, T}} \, \mathrm{d}s + e^{-A_T} \|\Psi\|_{\mathcal{B}_m} \right) \right] \\
    &\ge - \Big( (T-t)\|f\|_{\mathcal{B}_{m, T}} + \max\left(\|\Phi\|_{\mathcal{B}_{m, T}}, \|\Psi\|_{\mathcal{B}_m}\right) \Big) \E [ W^* ]. 
\end{align*}
By the moment estimate $P_{t,s}[w_m](x) \le K_m w_m(x)$, we have $\E[W^*] \le K_m w_m(x)$, thus:
\begin{equation*}
    J_\varepsilon(t, x; \gamma) \ge Z_{low}(t, x).
\end{equation*}
Taking the infimum over $\mathcal{V}_\varepsilon$ yields $u_\varepsilon(t, x) \ge Z_{low}(t, x)$.
\end{proof}
\subsection{Comparison Principle and Monotonicity}
\label{sec:comparison}

The monotonicity of the penalized solutions with respect to $\varepsilon$ is established here. Since the lack of global monotonicity in the mapping $u \mapsto -(u-\Phi)^+$ precludes a direct comparison via standard integral equations, we employ the shifted mild formulation from Theorem \ref{thm:equivalence_mild}. This approach allows the construction of an order-preserving operator, providing a rigorous basis for the monotonic convergence.

\begin{lemma}[Order-Preserving Property]
\label{lem:order_preserving}
Let Assumptions \ref{ass:coefficients} and \ref{ass:problem_data} hold. For any shift parameter satisfying $L \ge 1/\varepsilon$, the nonlinear shifted integral operator $\mathcal{M}_{\varepsilon, L}$ defined in Definition \ref{def:shifted_operator} is order-preserving on the measurable weighted space $\mathcal{B}_m([0, T] \times H)$. That is, if $u, v \in \mathcal{B}_m([0, T] \times H)$ satisfy $u(s, x) \le v(s, x)$ for all $(s, x) \in [0, T] \times H$, then $\mathcal{M}_{\varepsilon, L}(u)(s, x) \le \mathcal{M}_{\varepsilon, L}(v)(s, x)$ for all $(s, x) \in [0, T] \times H$.
\end{lemma}

\begin{proof}
For an arbitrary $(s, x) \in [0, T] \times H$, let $r_1 := u(s, x)$, $r_2 := v(s, x)$, and $c := \Phi(s, x)$. By hypothesis, $r_1 \le r_2$. Since the mapping $z \mapsto z^+$ is $1$-Lipschitz continuous, the difference of the shifted penalization terms satisfies:
\begin{align}
\label{eq:proof_shifted_diff}
    \Big( L r_2 - \frac{1}{\varepsilon} (r_2 - c)^+ \Big) - \Big( L r_1 - \frac{1}{\varepsilon} (r_1 - c)^+ \Big) &\ge L(r_2 - r_1) - \frac{1}{\varepsilon}|r_2 - r_1| = \left(L - \frac{1}{\varepsilon}\right)(r_2 - r_1).
\end{align}
Since $L \ge 1/\varepsilon$ and $r_1 \le r_2$, the right-hand side of \eqref{eq:proof_shifted_diff} is non-negative. Rearranging yields the pointwise ordering:
\begin{equation}
\label{eq:proof_pointwise_order}
    L u(s, x) - \frac{1}{\varepsilon} (u(s, x) - \Phi(s, x))^+ \le L v(s, x) - \frac{1}{\varepsilon} (v(s, x) - \Phi(s, x))^+.
\end{equation}
Adding the running cost $f$, applying the positivity-preserving transition evolution family $P_{t,s}$, and integrating against the exponential weight $e^{-L(s-t)}$ over $[t, T]$ preserves the inequality. Adding the identical terminal term $e^{-L(T-t)} P_{t,T}[\Psi](x)$ recovers the operator $\mathcal{M}_{\varepsilon, L}$, establishing $\mathcal{M}_{\varepsilon, L}(u)(s,x) \le \mathcal{M}_{\varepsilon, L}(v)(s,x)$ for all $(s,x)\in [0,T]\times H$, concluding the proof.
\end{proof}

Exploiting this order-preserving property within an iteration scheme, we can establish the monotonicity of the penalized solutions with respect to $\varepsilon$.

\begin{proposition}[Penalty Monotonicity]
\label{prop:penalty_monotonicity}
For any $0 < \varepsilon_1 < \varepsilon_2$, the corresponding unique mild solutions $u_{\varepsilon_1}$ and $u_{\varepsilon_2}$ of equation \eqref{eq:hjb_penalized} satisfy $u_{\varepsilon_1}(t, x) \le u_{\varepsilon_2}(t, x)$ for all $(t,x)\in[0, T] \times H$.
\end{proposition}

\begin{proof}
Set $L := 1/\varepsilon_1$. The condition $\varepsilon_1 < \varepsilon_2$ implies $L > 1/\varepsilon_2 > 0$. By Lemma \ref{lem:order_preserving}, this choice of $L$ ensures that both operators $\mathcal{M}_{\varepsilon_1, L}$ and $\mathcal{M}_{\varepsilon_2, L}$ are order-preserving on $\mathcal{B}_m([0, T] \times H)$. For any $u \in \mathcal{B}_m([0, T] \times H)$, since $(u-\Phi)^+ \ge 0$, the inequality $1/\varepsilon_1 > 1/\varepsilon_2$ implies:
\begin{equation*}
    -\frac{1}{\varepsilon_1} (u-\Phi)^+ \le -\frac{1}{\varepsilon_2} (u-\Phi)^+.
\end{equation*}
Adding the linear term $f + L u$, applying $P_{t,s}$, and integrating against the weight $e^{-L(s-t)}$ yields the global operator inequality $\mathcal{M}_{\varepsilon_1, L}(u) \le \mathcal{M}_{\varepsilon_2, L}(u)$. Defining the iteration sequences $u_i^{n+1} = \mathcal{M}_{\varepsilon_i, L}(u_i^n)$ with $u_1^0 = u_2^0 := e^{-L(T-t)} P_{t,T}[\Psi]$, the induction $u_1^n \le u_2^n$ holds. As $n \to \infty$, $u_i^n \to u_{\varepsilon_i}$ in $\mathcal{B}_m([0, T] \times H)$, establishing the result.
\end{proof}
\subsection{Mild Solutions for Variational Inequalities}
\label{sec:maximality}

To formalize the concept of a solution without requiring classical differentiability, we introduce the framework of mild sub-solutions and exact mild solutions for variational inequalities, adopting a notion similar to \cite[Definition 3.2]{FuMaTe17}.

\begin{definition}[Mild Solutions Framework]
\label{def:mild_subsolutions} 
\begin{enumerate}[label={(\textit{\roman*})}, leftmargin=*]
\item[]
 \item \emph{Mild Sub-solution:} A mapping $Z \in \mathcal{B}_m([0, T] \times H)$ is a \emph{mild sub-solution} to the parabolic variational inequality if it satisfies $Z(t, x) \le \Phi(t, x)$ on $[0, T) \times H$, $Z(T, x) \le \Psi(x)$ on $H$, and the dynamic sub-harmonicity condition for all $0 \le t \le s \le T$ and $x \in H$:
 \begin{equation} 
    \label{eq:subharmonicity_cond}
        Z(t, x) \le P_{t,s} [Z(s, \cdot)](x) + \int_t^s P_{t,r} [f(r, \cdot)](x) \, \mathrm{d}r.
    \end{equation}
 \item \emph{Exact Mild Solution:} A mapping $V \in \mathcal{B}_m([0, T] \times H)$ is the \emph{exact mild solution} to the variational inequality if it is a mild sub-solution and $Z(t, x) \le V(t, x)$ holds on $[0, T] \times H$ for any other mild sub-solution $Z$.
\end{enumerate}
\end{definition}

With the mild framework established, we demonstrate that the penalized approximants dominate any admissible mild sub-solution.
\begin{proposition}[Global Maximality of $u_\varepsilon$] 
\label{prop:global_maximality}
Let Assumptions \ref{ass:coefficients} and \ref{ass:problem_data} hold. If $Z:[0,T]\times H\to \mathbb{R}$ is a mild sub-solution of the variational inequality in the sense of Definition \ref{def:mild_subsolutions}, then for any penalization parameter $\varepsilon > 0$, the unique mild solution $u_\varepsilon$ to the penalized backward equation \eqref{eq:hjb_penalized} satisfies:
\begin{equation} 
\label{eq:maximality_eq}
    Z(t, x) \le u_\varepsilon(t, x), \quad \forall (t, x) \in [0, T] \times H.
\end{equation}
\end{proposition}

\begin{proof}
Evaluating the sub-harmonicity condition \eqref{eq:subharmonicity_cond} at $s = T$, applying the terminal constraint $Z(T, \cdot) \le \Psi(\cdot)$, and multiplying by the weight $e^{-\frac{T-t}{\varepsilon}}$ yields:
\begin{equation} 
\label{eq:max_boundary_eval}
    e^{-\frac{T-t}{\varepsilon}} Z(t, x) \le e^{-\frac{T-t}{\varepsilon}} P_{t,T}[\Psi](x) + e^{-\frac{T-t}{\varepsilon}} \int_t^T P_{t,r} [f(r, \cdot)](x) \, \mathrm{d}r.
\end{equation}
Integrating \eqref{eq:subharmonicity_cond} against the factor $\frac{1}{\varepsilon} e^{-\frac{s-t}{\varepsilon}}$ over $s \in [t, T]$ gives:
\begin{equation}
\label{eq:max_integrated_subharmonicity}
    Z(t, x) \left( 1 - e^{-\frac{T-t}{\varepsilon}} \right) \le \int_t^T \frac{1}{\varepsilon} e^{-\frac{s-t}{\varepsilon}} P_{t,s}[Z(s, \cdot)](x) \, \mathrm{d}s + \int_t^T \frac{1}{\varepsilon} e^{-\frac{s-t}{\varepsilon}} \left( \int_t^s P_{t,r}[f(r, \cdot)](x) \, \mathrm{d}r \right) \mathrm{d}s.
\end{equation}
Applying Fubini's Theorem to reverse the integration order in the double integral yields:
\begin{align}
\label{eq:max_fubini_eval}
    \int_t^T P_{t,r}[f(r, \cdot)](x) \left( \int_r^T \frac{1}{\varepsilon} e^{-\frac{s-t}{\varepsilon}} \, \mathrm{d}s \right) \mathrm{d}r &= \int_t^T e^{-\frac{r-t}{\varepsilon}} P_{t,r}[f(r, \cdot)](x) \, \mathrm{d}r - \int_t^T e^{-\frac{T-t}{\varepsilon}} P_{t,r}[f(r, \cdot)](x) \, \mathrm{d}r.
\end{align}
Adding \eqref{eq:max_boundary_eval} and \eqref{eq:max_integrated_subharmonicity} and applying \eqref{eq:max_fubini_eval} cancels the terminal source term. Relabeling the integration variables produces:
\begin{equation} 
\label{eq:max_consolidated}
    Z(t, x) \le e^{-\frac{T-t}{\varepsilon}} P_{t,T}[\Psi](x) + \int_t^T e^{-\frac{s-t}{\varepsilon}} P_{t,s} \left[ f(s, \cdot) + \frac{1}{\varepsilon} Z(s, \cdot) \right](x) \, \mathrm{d}s.
\end{equation}
Since $Z \le \Phi$, the positive part $(Z - \Phi)^+$ vanishes pointwise. Subtracting this zero term inside the integral identifies the right-hand side as the shifted operator $\mathcal{M}_{\varepsilon, 1/\varepsilon}$ evaluated at $Z$, yielding the operator inequality:
\begin{equation}
\label{eq:max_operator_ineq}
    Z \le \mathcal{M}_{\varepsilon, 1/\varepsilon}(Z).
\end{equation}
By Lemma \ref{lem:order_preserving}, $\mathcal{M}_{\varepsilon, 1/\varepsilon}$ is order-preserving on $\mathcal{B}_m([0, T] \times H)$. Iterating this inequality generates the monotonically bounded sequence $Z \le \mathcal{M}_{\varepsilon, 1/\varepsilon}^n(Z)$ for all $n \ge 1$. As $n \to \infty$, Theorem \ref{thm:existence_uniqueness} ensures $\mathcal{M}_{\varepsilon, 1/\varepsilon}^n(Z) \to u_\varepsilon$ in $\mathcal{B}_m([0, T] \times H)$. Passing to the pointwise limit establishes $Z(t, x) \le u_\varepsilon(t, x)$ on $[0, T] \times H$, concluding the proof.
\end{proof}

The following Proposition establishes the pointwise convergence of the approximants to a limit $V$, and shows that this limit belongs to the space $\mathcal{B}_m([0,T]\times H)$.

\begin{proposition}[Existence and Growth of the Limit]
\label{prop:existence_limit}
Let Assumptions \ref{ass:coefficients} and \ref{ass:problem_data} hold. The pointwise limit $V(t, x) := \lim_{\varepsilon \downarrow 0} u_\varepsilon(t, x)$ exists everywhere on $[0, T] \times H$ and belongs to the space $ \mathcal{B}_m([0, T] \times H)$.
\end{proposition}

\begin{proof}
By Proposition \ref{prop:penalty_monotonicity} and Theorem \ref{thm:uniform_lower_bound}, the sequence $\{u_\varepsilon\}_{\varepsilon > 0}$ decreases pointwise. Thus, the limit $V := \inf_{\varepsilon > 0} u_\varepsilon$ is well-defined on $[0, T] \times H$ and is measurable. Since $|u_\varepsilon| \le \max(|u_1|, |Z_{low}|)$ and both bounds belong to $\mathcal{B}_m([0,T]\times H)$, the limit inherits the uniform polynomial growth condition, ensuring $\|V\|_{\mathcal{B}_{m,T}} < \infty$
\end{proof}

Now, we establish the dynamic sub-harmonicity property for the pointwise limit $V$.
\begin{lemma}[Dynamic Sub-harmonicity]
\label{lem:subharmonicity_V}
Let Assumptions \ref{ass:coefficients} and \ref{ass:problem_data} hold. Then, the pointwise limit $V(t, x) := \lim_{\varepsilon \downarrow 0} u_\varepsilon(t, x)$ satisfies the dynamic sub-harmonicity condition for all $0 \le t \le s \le T$ and $x \in H$:
\begin{equation}
    V(t, x) \le P_{t,s} [V(s, \cdot)](x) + \int_t^s P_{t,r} [f(r, \cdot)](x) \, \mathrm{d}r.
\end{equation}
\end{lemma}

\begin{proof}
By the Markov property, the penalized solution satisfies:
\begin{equation*}
u_\varepsilon(t, x) = P_{t,s} [u_\varepsilon(s, \cdot)](x) + \int_t^s P_{t,r} \left[ f(r, \cdot) - \frac{1}{\varepsilon}(u_\varepsilon(r, \cdot) - \Phi(r, \cdot))^+ \right](x) \, \mathrm{d}r.
\end{equation*}
Since the penalization term $\frac{1}{\varepsilon}(u_\varepsilon - \Phi)^+$ is non-negative, we have the inequality:
\begin{equation*}
u_\varepsilon(t, x) \le P_{t,s} [u_\varepsilon(s, \cdot)](x) + \int_t^s P_{t,r} [f(r, \cdot)](x) \, \mathrm{d}r.
\end{equation*}
The uniform bound $\|u_\varepsilon\|_{\mathcal{B}_{m,T}} \le C$ and the polynomial growth stability of $P_{t,s}$ (Lemma \ref{lem:moment_estimates}) allow the application of the Dominated Convergence Theorem. Taking the limit as $\varepsilon \downarrow 0$ establishes the sub-harmonicity condition for $V$.
\end{proof}

The local Lipschitz continuity of the approximants, established in Theorem \ref{thm:uniform_lipschitz}, is inherited by the pointwise limit.

\begin{corollary}[Local Lipschitz Continuity of the pointwise limit]
\label{cor:value_function_lipschitz}
Let Assumptions \ref{ass:coefficients} and \ref{ass:problem_data} hold. Then, the pointwise limit $V(t, x) := \lim_{\varepsilon \downarrow 0} u_\varepsilon(t, x)$ is locally Lipschitz continuous in the space variable. Specifically, for all $x, y \in H$ and $t \in [0, T]$, the following bound holds:
\begin{equation}
\label{eq:value_function_lipschitz}
    |V(t, x) - V(t, y)| \le K(t) (1 + \|x\|_H^{m-1} + \|y\|_H^{m-1})\|x - y\|_H,
\end{equation}
where $K(t)$ is the continuous temporal envelope established in Theorem \ref{thm:uniform_lipschitz}.
\end{corollary}

\begin{proof}
By Proposition \ref{prop:existence_limit}, $V$ is the pointwise limit of the penalized approximants $u_\varepsilon$. The thesis follows from Theorem \ref{thm:uniform_lipschitz} by passing to the limit as $\varepsilon \downarrow 0$.
\end{proof}

To ensure uniform estimates on the penalty term independent of the penalization parameter $\varepsilon$, we define a specific class of admissible dynamic obstacles.

\begin{definition}[Admissible Obstacles]
\label{def:admissible_obstacle}
Let Assumptions \ref{ass:coefficients} and \ref{ass:problem_data} hold. Let $u_\varepsilon$ be the unique mild solution of equation \eqref{eq:hjb_penalized}. A dynamic obstacle $\Phi \in Lip_m([0, T] \times H)$ is said to be \emph{admissible} for the parabolic optimal stopping problem if the associated family of penalty terms, defined by:
\begin{equation}
\label{eq:yosida_penalty}
    \Pi_\varepsilon(t, x) := \frac{1}{\varepsilon} \left( u_\varepsilon(t, x) - \Phi(t, x) \right)^+,
\end{equation}
is uniformly bounded in the weighted parabolic topology. Specifically, there must exist a constant $K > 0$, independent of $\varepsilon > 0$, such that:
\begin{equation}
\label{eq:penalty_bound}
    \sup_{\varepsilon > 0} \|\Pi_\varepsilon\|_{\mathcal{B}_{m, T}} \le K.
\end{equation}
\end{definition}

\begin{remark}[Role of admissibility]
\label{rem:admissibility_forward_ref}
The definition is implicit, since it is formulated in terms of the solutions $u_\varepsilon$ of the penalized equations. Nevertheless, it captures the exact estimate needed in the verification and regularity arguments. The purpose of Section \ref{sec:admissible_obstacles} is to show that this condition is not merely formal: it follows from natural structural assumptions on the obstacle. In particular, admissibility holds when the obstacle is a suitable mild supersolution of the penalized dynamics, and this criterion can be checked in relevant classes of examples.
\end{remark}

Finally, the following lemma establishes the joint space-time continuity of the limit $V$.

\begin{lemma}[Joint Continuity of the pointwise limit]
\label{lem:joint_continuity_V}
Let Assumptions \ref{ass:coefficients} and \ref{ass:problem_data} hold. Assume that the obstacle $\Phi$ is admissible in the sense of Definition \ref{def:admissible_obstacle}. Then, the pointwise limit $V(t, x) := \lim_{\varepsilon \downarrow 0} u_\varepsilon(t, x)$ is jointly continuous on $[0, T] \times H$.
\end{lemma}

\begin{proof}
By Proposition \ref{prop:existence_limit}, the sequence of continuous mappings $\{u_\varepsilon\}_{\varepsilon > 0}$ decreases pointwise to $V$ as $\varepsilon \downarrow 0$. As the infimum of continuous functions, $V$ is upper semi-continuous on $[0, T] \times H$. To establish lower semi-continuity, let $(s_n, y_n) \to (t, x)$. The uniform spatial Lipschitz bound from Corollary \ref{cor:value_function_lipschitz} yields: $$|V(s_n, y_n) - V(s_n, x)| \le K(s_n)W(y_n, x) \to 0.$$ Thus, joint lower semi-continuity reduces to proving temporal lower semi-continuity. For the right limit, the dynamic sub-harmonicity condition (Lemma \ref{lem:subharmonicity_V}) combined with the spatial Lipschitz continuity of $V$ implies:
\begin{equation*}
    V(t, x) \le V(s, x) + K(s)\E[W(X(s; t, x), x)] + \int_t^s P_{t,r}[f(r, \cdot)](x) \, \mathrm{d}r.
\end{equation*}
As $s \downarrow t$, the mean-square continuity of the mild stochastic flow and the polynomial growth of $f$ ensure the error terms vanish, yielding $\liminf_{s \downarrow t} V(s, x) \ge V(t, x)$. For the left limit, using the definition of mild solution for $u_\varepsilon$ and applying the uniform admissibility bound $\Pi_\varepsilon \le K w_m$ provides:
\begin{equation*}
    u_\varepsilon(s, x) \ge P_{s,t}[u_\varepsilon(t, \cdot)](x) + \int_s^t P_{s,r}\left[f(r, \cdot) - K w_m(\cdot)\right](x) \, \mathrm{d}r.
\end{equation*}
Taking $\varepsilon \downarrow 0$ via monotone convergence, and expanding $P_{s,t}[V(t, \cdot)](x)$ via the spatial Lipschitz property, we obtain:
\begin{equation*}
    V(s, x) \ge V(t, x) - K(t)\E[W(X(t; s, x), x)] + \int_s^t P_{s,r}\left[f(r, \cdot) - K w_m(\cdot)\right](x) \, \mathrm{d}r.
\end{equation*}
Taking $s \uparrow t$ yields $\liminf_{s \uparrow t} V(s, x) \ge V(t, x)$, which concludes the proof.
\end{proof}

We now prove that the pointwise limit $V$ is the exact mild solution to the variational inequality.

\begin{theorem}[Convergence to the Exact Mild Solution]
\label{thm:convergence_exact}
Let Assumptions \ref{ass:coefficients} and \ref{ass:problem_data} hold. Assume that the obstacle $\Phi$ is admissible in the sense of Definition \ref{def:admissible_obstacle}. Then, the pointwise limit $V(t, x) := \lim_{\varepsilon \downarrow 0} u_\varepsilon(t, x)$ is the unique exact mild solution to the variational inequality in the sense of Definition \ref{def:mild_subsolutions}.
\end{theorem}

\begin{proof}
The proof relies on verifying the requirements of Definition \ref{def:mild_subsolutions} and proceeds in three steps.

\emph{Step 1: Growth conditions and dynamic sub-harmonicity.} 
By Proposition \ref{prop:existence_limit}, the pointwise limit $V$ exists, inherits the uniform polynomial growth bounds, and belongs to the weighted measurable space $\mathcal{B}_m([0, T] \times H)$. Furthermore, Lemma \ref{lem:subharmonicity_V} establishes that $V$ satisfies the dynamic sub-harmonicity condition associated with the transition evolution family.

\emph{Step 2: Geometric constraint.}
Multiplying the integral formulation of the penalized equation for $u_\varepsilon$ by $\varepsilon$ yields:
\begin{equation*}
    \int_t^T P_{t,s}\left[ (u_\varepsilon(s, \cdot) - \Phi(s, \cdot))^+ \right](x) \, \mathrm{d}s = \varepsilon \left( P_{t,T}[\Psi](x) + \int_t^T P_{t,s} [f(s, \cdot)](x) \, \mathrm{d}s - u_\varepsilon(t, x) \right).
\end{equation*}
Since the right-hand side is locally bounded, taking $\varepsilon \downarrow 0$ vanishes the expression. The uniform polynomial bounds on $u_\varepsilon$ justify the application of the Dominated Convergence Theorem, providing:
\begin{equation*}
    \int_t^T \E \left[ \big(V(s, X(s; t, x)) - \Phi(s, X(s; t, x))\big)^+ \right] \mathrm{d}s = 0.
\end{equation*}
By Lemma \ref{lem:joint_continuity_V}, the limit mapping $V$ is jointly continuous on $[0, T] \times H$. Since the obstacle $\Phi$ and the stochastic flow $X(\cdot; t, x)$ are also continuous, the mapping $s \mapsto \E[\big(V(s, X(s; t, x)) - \Phi(s, X(s; t, x))\big)^+]$ is non-negative and continuous on $[t, T]$. For the integral of a continuous, non-negative function to be exactly zero, the integrand must vanish identically on $[t, T]$. Evaluating at the initial time $s = t$ guarantees $(V(t, x) - \Phi(t, x))^+ = 0$, providing the pointwise constraint $V(t, x) \le \Phi(t, x)$ on $[0, T) \times H$. Moreover, the penalized terminal condition $u_\varepsilon(T, \cdot) = \Psi(\cdot)$ passes to the limit as $V(T, \cdot) = \Psi(\cdot)$. By Assumption \ref{ass:problem_data}, $\Psi \le \Phi(T, \cdot)$, ensuring the constraint is globally satisfied.

\emph{Step 3: Global maximality and uniqueness.}
To conclude that $V$ is the exact mild solution, we must prove it dominates any other mild sub-solution. Let $Z \in \mathcal{B}_m([0, T] \times H)$ be an arbitrary mild sub-solution. By Proposition \ref{prop:global_maximality}, the penalized approximants dominate $Z$, meaning:
\begin{equation*}
    Z(t, x) \le u_\varepsilon(t, x), \quad \forall (t, x) \in [0, T] \times H, \quad \forall \varepsilon > 0.
\end{equation*}
Passing to the pointwise limit as $\varepsilon \downarrow 0$ preserves the inequality, yielding $Z(t, x) \le V(t, x)$ for all $(t, x) \in [0, T] \times H$. This confirms that $V$ is the maximal mild sub-solution. Uniqueness follows immediatly, since any two exact mild solutions must mutually dominate each other, forcing them to coincide.
\end{proof}

\subsection{Identification with the Value Function}
\noindent We provide the definitive connection between the limit of the penalized family $u_\varepsilon$ and the probabilistic value function. By employing a martingale verification argument (see, e.g., \cite{BeLi82} for a similar proof in a finite-dimensional setting), we show that the limit map satisfies the optimality principle, thereby identifying it as the value function of the optimal stopping problem. Let $V = \lim_{\varepsilon \downarrow 0} u_\varepsilon$ denote the exact mild solution established in Theorem \ref{thm:convergence_exact}.

We begin by establishing that the exact mild solution provides a lower bound for the expected cost of any admissible stopping strategy.
\begin{proposition}[Lower Bound and Submartingale Property]
\label{prop:lower_bound}
Let Assumptions \ref{ass:coefficients} and \ref{ass:problem_data} hold. Assume that the obstacle $\Phi$ is admissible in the sense of Definition \ref{def:admissible_obstacle}. Let $V$ be the exact mild solution to the parabolic variational inequality in the sense of Definition \ref{def:mild_subsolutions}. For any initial state $(t, x) \in [0, T] \times H$ and any admissible stopping time $\tau \in \mathcal{T}_{t, T}$, the expected cost functional $J$ defined in \eqref{eq:def_cost} satisfies:
\begin{equation}
    V(t, x) \le J(t, x; \tau).
\end{equation}
\end{proposition}

\begin{proof}
By Theorem \ref{thm:convergence_exact}, $V$ is a parabolic mild sub-solution. Along the stochastic flow, we define the compensated process:
\begin{equation}
\label{eq:proof_martingale_M}
    M_s := \int_t^s f(r, X(r)) \, \mathrm{d}r + V(s, X(s)), \quad s \in [t, T].
\end{equation} 
Assumption \ref{ass:problem_data}, the moment estimates from Lemma \ref{lem:moment_estimates} together with the growth condition obtained in Theorem \ref{thm:convergence_exact} ensure the integrability of $M_s$. For $t \le u \le s \le T$, the Markov property yields:
\begin{align}
\label{eq:proof_M_cond_exp}
    \E [M_s \mid \mathcal{F}_u] &= \int_t^u f(r, X(r)) \, \mathrm{d}r + \int_u^s P_{u,r}[f(r, \cdot)](X(u)) \, \mathrm{d}r + P_{u,s}[V(s, \cdot)](X(u)).
\end{align}
The dynamic sub-harmonicity of $V$ over $[u, s]$ implies that the last two terms in \eqref{eq:proof_M_cond_exp} are bounded below by $V(u, X(u))$, proving that $\E [M_s \mid \mathcal{F}_u] \ge M_u$. Thus, $M_s$ is a continuous $\F$-submartingale. For any admissible stopping strategy $\tau \in \mathcal{T}_{t, T}$, Doob's Optional Stopping Theorem provides $V(t, x) = M_t \le \E[M_\tau]$, which expands to:
\begin{equation}
\label{eq:proof_doob_eval}
    V(t, x) \le \E \left[ \int_t^\tau f(r, X(r)) \, \mathrm{d}r + V(\tau, X(\tau)) \right].
\end{equation}
Since $V \le \Phi$ on $[0, T)$ and $V(T, \cdot) = \Psi(\cdot)$, the terminal term satisfies:
\begin{equation}
\label{eq:proof_V_obstacle_bound}
 V(\tau, X(\tau)) \le \Phi(\tau, X(\tau))\ind_{\{\tau < T\}} + \Psi(X(T))\ind_{\{\tau = T\}}.
\end{equation}
Substituting \eqref{eq:proof_V_obstacle_bound} into \eqref{eq:proof_doob_eval} recovers the expected cost functional:
\begin{equation*}
 V(t, x) \le \E \left[ \int_t^\tau f(r, X(r)) \, \mathrm{d}r + \Phi(\tau, X(\tau))\ind_{\{\tau < T\}} + \Psi(X(T))\ind_{\{\tau = T\}} \right] = J(t, x; \tau).
\end{equation*}
This establishes the lower bound for all $\tau \in \mathcal{T}_{t, T}$.
\end{proof}

We now identify the exact mild solution with the optimal stopping value function and construct the optimal stopping time.
\begin{theorem}[Optimal Strategy and Value Function Identity]
\label{thm:optimal_strategy}
Let Assumptions \ref{ass:coefficients} and \ref{ass:problem_data} hold, and suppose the dynamic obstacle $\Phi$ is admissible in the sense of Definition \ref{def:admissible_obstacle}. Let $V$ be the exact mild solution to the parabolic variational inequality in the sense of Definition \ref{def:mild_subsolutions}. Then, $V$ coincides with the value function of the optimal stopping problem:
\begin{equation}
\label{eq:value_function_identity}
    V(t, x) = \inf_{\tau \in \mathcal{T}_{t, T}} J(t, x; \tau), \quad \forall (t, x) \in [0, T] \times H.
\end{equation}
Furthermore, the hitting time of the contact set, defined as
\begin{equation}
\label{eq:optimal_stopping_time}
    \tau^* := \inf \{ s \in [t, T] : V(s, X(s; t, x)) = \Phi(s, X(s; t, x)) \} \wedge T,
\end{equation}
is an optimal stopping strategy, yielding $V(t, x) = J(t, x; \tau^*)$.
\end{theorem}

\begin{proof}
Let $\mathcal{C} := \{ (s, y) \in [0, T) \times H : V(s, y) < \Phi(s, y) \}$ be the continuation region. Consider the closed subsets
\begin{equation*}
    \mathcal{C}_n := \left\{ (s, y) \in [0, T) \times H : \Phi(s, y) - V(s, y) \ge \frac{1}{n} \right\},
\end{equation*}
and the associated exit times $\sigma_n := \inf\{s \ge t : (s, X(s)) \notin \mathcal{C}_n\} \wedge T$. As established in the Proof of Theorem \ref{thm:feynman_kac} for $\varepsilon > 0$, the compensated process
\begin{equation*}
    N^\varepsilon_s := \int_t^s \left( f(r, X(r)) - \frac{1}{\varepsilon}(u_\varepsilon(r, X(r)) - \Phi(r, X(r)))^+ \right) \mathrm{d}r + u_\varepsilon(s, X(s))
\end{equation*}
is a continuous $\F$-martingale. By Doob's Optional Stopping Theorem at $\sigma_n$:
\begin{equation*} 
    u_\varepsilon(t, x) + \E \left[ \int_t^{\sigma_n} \frac{\left( u_\varepsilon(r, X(r)) - \Phi(r, X(r)) \right)^+}{\varepsilon} \, \mathrm{d}r \right] = \E \left[ \int_t^{\sigma_n} f(r, X(r)) \, \mathrm{d}r + u_\varepsilon(\sigma_n, X(\sigma_n)) \right].
\end{equation*}
Fix a sample path $\omega \in \Omega$. On the compact interval $I := [t, \sigma_n(\omega)]$, we have $V(s, X(s)) < \Phi(s, X(s))$ by construction. Since $u_\varepsilon \downarrow V$ pointwise, for each $s \in I$ there exists $\varepsilon_s > 0$ such that $u_{\varepsilon_s}(s, X(s)) < \Phi(s, X(s))$. The joint continuity of $u_{\varepsilon_s}$ and $\Phi$ in $\mathcal{C}_m([0, T] \times H)$, alongside the pathwise continuity of the state trajectories, ensures that $u_{\varepsilon_s}(r, X(r)) < \Phi(r, X(r))$ on an open neighborhood $O_s$ of $s$ in $I$. By the compactness of $I$, there exists a finite subcover $O_{s_1}, \dots, O_{s_k}$. Setting $\bar{\varepsilon} := \min_{i} \varepsilon_{s_i} > 0$, the monotonicity in Proposition \ref{prop:penalty_monotonicity} ensures that for all $\varepsilon \le \bar{\varepsilon}$ and $r \in I$:
\begin{equation*}
    u_\varepsilon(r, X(r)) \le u_{\bar{\varepsilon}}(r, X(r)) < \Phi(r, X(r)).
\end{equation*}
Thus, $\mathbb{P}$-a.s., the penalization term vanishes for sufficiently small $\varepsilon$. Since the obstacle $\Phi$ is admissible in the sense of Definition \ref{def:admissible_obstacle}, we can apply the Dominated Convergence Theorem, obtaining:
\begin{equation*}
    V(t, x) = \E \left[ \int_t^{\sigma_n} f(r, X(r)) \, \mathrm{d}r + V(\sigma_n, X(\sigma_n)) \right].
\end{equation*}
As $n \to \infty$, $\sigma_n \uparrow \tau^*$ $\mathbb{P}$-a.s. By the pathwise continuity of the solution, Lemma \ref{lem:joint_continuity_V} and a second application of the Dominated Convergence Theorem:
\begin{equation*}
    V(t, x) = \E \left[ \int_t^{\tau^*} f(r, X(r)) \, \mathrm{d}r + V(\tau^*, X(\tau^*)) \right].
\end{equation*}
Decomposing the terminal term $V(\tau^*, X(\tau^*))$ gives:
\begin{equation*}
    V(\tau^*, X(\tau^*)) = V(\tau^*, X(\tau^*))\ind_{\{\tau^* < T\}} + V(T, X(T))\ind_{\{\tau^* = T\}}.
\end{equation*}
Since $V(\tau^*, X(\tau^*)) = \Phi(\tau^*, X(\tau^*))$ on $\{\tau^* < T\}$ and $V(T, X(T)) = \Psi(X(T))$, we recover the cost functional:
\begin{equation*}
    V(t, x) = \E \left[ \int_t^{\tau^*} f(r, X(r)) \, \mathrm{d}r + \Phi(\tau^*, X(\tau^*))\ind_{\{\tau^* < T\}} + \Psi(X(T))\ind_{\{\tau^* = T\}} \right] = J(t, x; \tau^*).
\end{equation*}
Given $V(t, x) \le J(t, x; \tau)$ for all $\tau \in \mathcal{T}_{t, T}$ from Proposition \ref{prop:lower_bound}, it follows that $V(t, x) = \inf_{\tau} J(t, x; \tau)$.
\end{proof}

\section{Regularity of Solutions}
\label{sec:verification}
\noindent In this section we address the spatial regularity of the value function. By exploiting the intrinsic smoothing properties of the transition semigroup, we upgrade the solution's regularity from mere continuity to continuous Fréchet differentiability and local Hölder continuity of the gradient. In the language of optimal stopping, this yields the corresponding smooth-fit property.

\subsection{Semigroup Regularization}
\label{sec:semigroup}

The regularizing properties of the transition semigroup are described by the following qualitative differentiability assumption and associated quantitative gradient estimates. Let $\{P_{t,s}\}_{0\leq t \leq s \leq T}$ be the two parameter transition evolution family defined in Definition \ref{def:evolution_operators}. 

\begin{assumption}[Global Smoothing Properties]
\label{ass:smoothing}
\begin{enumerate}[label={(\textit{\roman*})}, leftmargin=*]
\item[]
    \item \emph{Spatial Smoothing.} For any $0 \le t < s \le T$ and $\varphi \in \mathcal{B}_m(H)$, the mapping $x \mapsto P_{t,s}[\varphi](x)$ is twice continuously Fréchet differentiable on $H$.
    
    \item \emph{Singular Gradient Estimates.} There exist constants $C_1 > 0$ and $\gamma \in (0, 1)$ such that for all $x \in H$, $0 \le t < s \le T$, and $\varphi \in \mathcal{B}_m(H)$, the spatial Fréchet derivatives satisfy:
    \begin{gather}
        \|\nabla P_{t,s}[\varphi](x)\|_H \le C_1 (s-t)^{-\gamma} w_m(x) \|\varphi\|_{\mathcal{C}_m}, \label{eq:grad1_bound} \\
        \|\nabla^2 P_{t,s}[\varphi](x)\|_{\mathcal{L}(H)} \le C_1 (s-t)^{-2\gamma} w_m(x) \|\varphi\|_{\mathcal{C}_m}. \label{eq:grad2_bound}
    \end{gather}
\end{enumerate}
\end{assumption}

\begin{remark}
The estimates in \eqref{eq:grad1_bound} and \eqref{eq:grad2_bound} allow the spatial derivatives of $P_{t,s}[\varphi]$ to become singular as $s \downarrow t$, which is unavoidable even in classical parabolic problems. The relevant point is that the singularity is integrable in time. The condition $\gamma<1$ ensures that the first-order estimate can be integrated in the mild representation, while the stronger restriction $\gamma(1+\beta)<1$ appearing below is precisely the condition under which the H\"older estimate for the gradient is also integrable. In this sense, the exponent $\gamma$ measures the strength of the semigroup regularization: the smaller $\gamma$ is, the stronger the smoothing effect and the larger the admissible H\"older exponents for the gradient.
\end{remark}

By interpolating the singular bounds of the first and second Fréchet derivatives, we can establish the local equi-Hölder continuity of the semigroup gradient.
\begin{lemma}[Hölder Regularity of the Transitioned Source]
\label{lem:spatial_transition_reg}
Let Assumptions \ref{ass:coefficients} and \ref{ass:smoothing} hold, and let $\varphi \in \mathcal{B}_m(H)$. For any $0 \le t < s \le T$ and any Hölder exponent $\beta \in (0, 1)$, the Fréchet gradient of the mapping $x \mapsto P_{t,s}[\varphi](x)$ is locally equi-Hölder continuous. Specifically, there exists a constant $C_3 > 0$ such that for all $x, y \in H$:
\begin{equation}
\label{eq:transition_holder_bound}
    \|\nabla P_{t,s}[\varphi](x) - \nabla P_{t,s}[\varphi](y)\|_H \le C_3 (s-t)^{-\gamma(1+\beta)} \Big(1 + \sup_{z \in [x,y]} \|z\|_H^m\Big) \|\varphi\|_{\mathcal{B}_m} \|x - y\|_H^\beta.
\end{equation}
\end{lemma}

\begin{proof}
The proof exploits a standard interpolation technique, balancing the uniform variation against the Lipschitz bound provided by the Mean Value Theorem. Let $x, y \in H$. By applying the triangle inequality alongside the first-order singular estimate \eqref{eq:grad1_bound} from Assumption \ref{ass:smoothing}, we obtain the uniform variation bound:
\begin{equation} 
\label{eq:lem_grad_variation_1}
    \| \nabla P_{t,s}[\varphi](x) - \nabla P_{t,s}[\varphi](y) \|_H \le 2 \sup_{z \in [x,y]} \|\nabla P_{t,s}[\varphi](z)\|_H \le 2 C_1 (s-t)^{-\gamma} \Big(1 + \sup_{z \in [x,y]} \|z\|_H^m\Big) \|\varphi\|_{\mathcal{B}_m}.
\end{equation}
Alternatively, applying the Mean Value Theorem and the second-order estimate \eqref{eq:grad2_bound} yields the Lipschitz bound:
\begin{align} 
\label{eq:lem_grad_variation_2}
    \| \nabla P_{t,s}[\varphi](x) - \nabla P_{t,s}[\varphi](y) \|_H &\le \left( \sup_{z \in [x,y]} \|\nabla^2 P_{t,s}[\varphi](z)\|_{\mathcal{L}(H)} \right) \|x - y\|_H \nonumber \\
    &\le C_1 (s-t)^{-2\gamma} \Big(1 + \sup_{z \in [x,y]} \|z\|_H^m\Big) \|\varphi\|_{\mathcal{B}_m} \|x - y\|_H.
\end{align}
Interpolating the bounds \eqref{eq:lem_grad_variation_1} and \eqref{eq:lem_grad_variation_2} with exponent $\beta \in (0, 1)$ gives:
\begin{equation}
    \| \nabla P_{t,s}[\varphi](x) - \nabla P_{t,s}[\varphi](y) \|_H \le \left( 2 \sup_{z\in [x,y]} \|\nabla P_{t,s}[\varphi](z)\|_H \right)^{1-\beta} \left( \sup_{z\in [x,y]} \|\nabla^2 P_{t,s}[\varphi](z)\|_{\mathcal{L}(H)} \|x - y\|_H \right)^\beta. 
\end{equation}
Setting $C_3 := 2^{1-\beta} C_1$ and substituting the singular estimates into this interpolation yields exactly the bound \eqref{eq:transition_holder_bound}, concluding the proof.
\end{proof}

We now show that temporal convolutions inherit continuous differentiability and local Hölder regularity of the gradient.
\begin{proposition}[Spatial Regularity of Convolutions]
\label{prop:spatial_conv}
Let Assumptions \ref{ass:coefficients} and \ref{ass:smoothing} hold, and let $f \in \mathcal{B}_m([0, T] \times H)$. For any $t \in [0, T)$, the temporal convolution
\begin{equation}
    g(t, x) := \int_t^T P_{t,s}[f(s, \cdot)](x) \, \mathrm{d}s
\end{equation}
satisfies the following regularity properties on $H$:
\begin{enumerate}[label={(\textit{\roman*})}, leftmargin=*]
\item The mapping $g(t, \cdot)$ is continuously Fréchet differentiable on $H$. The spatial differential operator commutes with the temporal integral, and there exists a constant $C_2 > 0$ such that the following bound holds:
    \begin{equation}\label{eq:spatial_conv_grad} 
        \|\nabla g(t, x)\|_H \le C_2 \frac{(T-t)^{1-\gamma}}{1-\gamma} w_m(x) \|f\|_{\mathcal{B}_{m, T}}.
    \end{equation}
    
    \item The spatial gradient $x \mapsto \nabla g(t, x)$ is locally equi-Hölder continuous. Specifically, for any Hölder exponent $\beta \in (0, 1)$ satisfying $\gamma(1+\beta) < 1$, the following estimate holds for all $x, y \in H$:
    \begin{equation}
        \|\nabla g(t, x) - \nabla g(t, y)\|_H \le C_3 \frac{(T-t)^{1-\gamma(1+\beta)}}{1-\gamma(1+\beta)} \Big(1 + \sup_{z \in [x,y]} \|z\|_H^m\Big) \|f\|_{\mathcal{B}_{m, T}} \|x - y\|_H^{\beta}.
    \end{equation}
\end{enumerate}
\end{proposition}

\begin{proof} The proof proceeds in two main steps, addressing each claim of the Proposition.
\begin{enumerate}[label={(\textit{\roman*})}, leftmargin=*]
\item Under Assumption \ref{ass:smoothing}, $x \mapsto P_{t,s}[f(s, \cdot)](x)$ is Fréchet differentiable for $s \in (t, T]$. The first-order singular estimate \eqref{eq:grad1_bound} provides the local uniform bound $C_1 w_m(x) \|f\|_{\mathcal{C}_{m, T}} (s-t)^{-\gamma}$. Since $\gamma \in (0, 1)$, this bound is Lebesgue-integrable over $[t, T]$. This justifies the commutation of the gradient and the parameter-dependent integral by the Dominated Convergence Theorem:
\begin{align}
    \|\nabla g(t, x)\|_H &\le \int_t^T \|\nabla P_{t,s}[f(s, \cdot)](x)\|_H \, \mathrm{d}s \nonumber \\ &\le C_1 w_m(x) \|f\|_{\mathcal{B}_{m, T}} \int_t^T (s-t)^{-\gamma} \, \mathrm{d}s = \frac{C_1 (T-t)^{1-\gamma}}{1-\gamma} w_m(x) \|f\|_{\mathcal{B}_{m, T}},
\end{align}
and we obtain the thesis defining $C_{2}:=C_{1}$.

\item For $x, y \in H$, the spatial variation of the gradient is bounded by integrating the Hölder estimate obtained in Lemma \ref{lem:spatial_transition_reg}:
\begin{equation*}
    \|\nabla g(t, x) - \nabla g(t, y)\|_H \le \int_t^T \| \nabla P_{t,s}[f(s,\cdot)](x) - \nabla P_{t,s}[f(s,\cdot)](y) \|_H \, \mathrm{d}s.
\end{equation*}
Substituting the explicit bound \eqref{eq:transition_holder_bound} yields:
\begin{equation*}
    \|\nabla g(t, x) - \nabla g(t, y)\|_H \le C_3 \Big(1 + \sup_{z \in [x,y]} \|z\|_H^m\Big) \|f\|_{\mathcal{B}_{m, T}} \|x - y\|_H^\beta \int_t^T (s-t)^{-\gamma(1+\beta)} \, \mathrm{d}s.
\end{equation*}
The temporal integral converges if and only if $\gamma(1+\beta) < 1$. Evaluating the integral confirms the desired local equi-Hölder continuity.
\end{enumerate}
\end{proof}

\subsection{Regularity of the Penalized Solutions}
\label{sec:frechet}

We prove the continuous spatial Fréchet differentiability of the parabolic approximants $u_\varepsilon(t, \cdot)$ on the Hilbert space $H$. The existence of these spatial gradients is deduced from the regularizing properties of the transition semigroup $P_{t,s}$ operating within the space $\mathcal{B}_m([0, T] \times H)$.

We first establish the regularity of the penalized approximants, provided the transition semigroup satisfies the required smoothing assumptions.
\begin{lemma}[Global Regularity and Gradient Estimates]
\label{lem:global_reg_approx}
Let Assumptions \ref{ass:smoothing}, \ref{ass:coefficients}, and \ref{ass:problem_data} hold. For any $\varepsilon > 0$ and $t \in [0, T)$, the mild solution $u_\varepsilon \in \mathcal{C}_m([0, T] \times H)$ to equation \eqref{eq:hjb_penalized} satisfies:
\begin{enumerate}[label={(\textit{\roman*})}, leftmargin=*]
    \item \emph{Continuous Differentiability.} The mapping $u_\varepsilon(t, \cdot)$ is continuously Fréchet differentiable on $H$. Defining the penalty term $\Pi_\varepsilon := \frac{1}{\varepsilon}(u_\varepsilon - \Phi)^+$, the spatial gradient satisfies:
    \begin{equation}
    \label{eq:exact_grad_bound}
        \|\nabla u_\varepsilon(t, x)\|_H \le C_1 \left[ \|\Psi\|_{\mathcal{B}_m} (T-t)^{-\gamma} + \big( \|f\|_{\mathcal{B}_{m, T}} + \|\Pi_\varepsilon\|_{\mathcal{B}_{m, T}} \big) \frac{(T-t)^{1-\gamma}}{1-\gamma} \right] w_m(x), \quad \forall x \in H.
    \end{equation}
    
    \item \emph{Local Equi-Hölder Continuity.} For any radius $R > 0$ and exponent $\beta \in (0, 1)$ satisfying $\gamma(1+\beta) < 1$, the spatial gradient is locally $\beta$-Hölder continuous. Setting $C_R := C_{3}(1 + R^m)$, for all $x, y \in B_R \subset H$:
    \begin{equation}
    \label{eq:exact_holder_grad}
        \|\nabla u_\varepsilon(t, x) - \nabla u_\varepsilon(t, y)\|_H \le C_R \left[ \|\Psi\|_{\mathcal{B}_m} (T-t)^{-\gamma(1+\beta)} + \big( \|f\|_{\mathcal{B}_{m, T}} + \|\Pi_\varepsilon\|_{\mathcal{B}_{m, T}} \big) \frac{(T-t)^{1-\gamma(1+\beta)}}{1-\gamma(1+\beta)} \right] \|x - y\|_H^\beta.
    \end{equation}
\end{enumerate}
\end{lemma}

\begin{proof}
The proof relies on differentiating the direct mild integral equation and controlling the resulting singular integrals. We analyze the direct formulation governing the penalized approximant (Definition \ref{def:mild_formulation}):
\begin{equation} 
\label{eq:direct_integral_rep_u_eps}
    u_\varepsilon(t, x) = P_{t,T}[\Psi](x) + \int_t^T P_{t,s} [H_\varepsilon(s, \cdot)](x) \, \mathrm{d}s,
\end{equation}
where the effective source term is defined as $H_\varepsilon(s, x) := f(s, x) - \Pi_\varepsilon(s, x)$. Since $f, \Phi, u_\varepsilon \in \mathcal{B}_m([0, T] \times H)$, the penalty term $\Pi_\varepsilon$ and the source $H_\varepsilon$ belong to the same space. By the triangle inequality, we have the norm bound $\|H_\varepsilon\|_{\mathcal{B}_{m, T}} \le \|f\|_{\mathcal{B}_{m, T}} + \|\Pi_\varepsilon\|_{\mathcal{B}_{m, T}}$.

\begin{enumerate}[label={(\textit{\roman*})}, leftmargin=*]
    \item \emph{Proof of continuous differentiability.} Under Assumption \ref{ass:smoothing}, the spatial gradient of the transitioned source satisfies:
    \begin{equation*}
        \|\nabla P_{t,s} [H_\varepsilon(s, \cdot)](y)\|_H \le C_1 (s-t)^{-\gamma} w_m(y) \|H_\varepsilon(s, \cdot)\|_{\mathcal{B}_m}.
    \end{equation*}
    For any fixed $x \in H$, $w_m$ is locally bounded on the closed ball $B_1(x)$. Since $\gamma \in (0, 1)$, the singularity is Lebesgue-integrable over $[t, T]$. Combined with the terminal bound $\|\nabla P_{t,T}[\Psi](x)\|_H \le C_1 (T-t)^{-\gamma} w_m(x) \|\Psi\|_{\mathcal{B}_m}$, we justify the commutation of the Fréchet differential with the integral:
    \begin{equation} 
    \label{eq:direct_bochner_grad_u_eps}
        \nabla u_\varepsilon(t, x) = \nabla P_{t,T}[\Psi](x) + \int_t^T \nabla P_{t,s} [H_\varepsilon(s, \cdot)](x) \, \mathrm{d}s.
    \end{equation}
    Taking the $H$-norm, factoring out the polynomial weight $w_m(x)$, and integrating the temporal singularity $\int_t^T (s-t)^{-\gamma} \, \mathrm{d}s = \frac{1}{1-\gamma}(T-t)^{1-\gamma}$ directly yields the pointwise estimate \eqref{eq:exact_grad_bound}.

    \item \emph{Proof of local equi-Hölder continuity.} Let $x, y \in B_R \subset H$. The spatial variation of the terminal term is bounded via the interpolation estimates from Lemma \ref{lem:spatial_transition_reg}. Observing that $\sup_{z \in [x,y]} w_m(z) \le 1 + R^m$, we have:
    \begin{equation*}
        \| \nabla P_{t,T}[\Psi](x) - \nabla P_{t,T}[\Psi](y) \|_H \le C_R \|\Psi\|_{\mathcal{C}_m} (T-t)^{-\gamma(1+\beta)} \|x-y\|_H^\beta.
    \end{equation*}
    Similarly, the variation of the integral component is controlled by integrating the Hölder bound of the transitioned source over $[t, T]$:
    \begin{equation*}
        \int_t^T \| \nabla P_{t,s}[H_\varepsilon](x) - \nabla P_{t,s}[H_\varepsilon](y) \|_H \, \mathrm{d}s \le C_R \|H_\varepsilon\|_{\mathcal{B}_{m, T}} \|x-y\|_H^\beta \int_t^T (s-t)^{-\gamma(1+\beta)} \, \mathrm{d}s.
    \end{equation*}
    The condition $\gamma(1+\beta) < 1$ ensures the convergence of the temporal integral to $\frac{1}{1-\gamma(1+\beta)}(T-t)^{1-\gamma(1+\beta)}$. Summing the terminal and integral bounds and isolating the common factor $\|x-y\|_H^\beta$ establishes the Hölder bound \eqref{eq:exact_holder_grad}, completing the proof.
\end{enumerate}
\end{proof}

By exploiting the admissibility of the dynamic obstacle, we can now extract gradient estimates that are entirely independent of the penalization parameter.
\begin{corollary}[Uniform Gradient Estimates] 
\label{cor:uniform_gradients}
Let Assumptions \ref{ass:smoothing}, \ref{ass:coefficients}, and \ref{ass:problem_data} hold. Let the dynamic obstacle $\Phi$ be admissible in the sense of Definition \ref{def:admissible_obstacle}, and let $u_\varepsilon$ be the unique mild solution of equation \eqref{eq:hjb_penalized}. For any $t \in [0, T)$, the family of continuous Fréchet gradients $\{\nabla u_\varepsilon(t, \cdot)\}_{\varepsilon > 0}$ satisfies the following uniform bounds:
\begin{enumerate}[label={(\textit{\roman*})}, leftmargin=*]
    \item \emph{Uniform Boundedness.} There exists a continuous temporal envelope $K_1(t) > 0$, independent of $\varepsilon$, such that:
    \begin{equation}
    \label{eq:uniform_grad_bound}
        \sup_{\varepsilon > 0} \|\nabla u_\varepsilon(t, x)\|_H \le K_1(t) w_m(x), \quad \forall x \in H.
    \end{equation}
    
    \item \emph{Uniform Equi-Hölder Continuity.} For any closed ball $B_R \subset H$ of radius $R$, and any exponent $\beta \in (0, 1)$ satisfying $\gamma(1+\beta) < 1$, there exists a continuous temporal envelope $L_{R, T}(t) > 0$, independent of $\varepsilon$, such that:
    \begin{equation}
    \label{eq:equi_holder_grad}
        \sup_{\varepsilon > 0} \|\nabla u_\varepsilon(t, x) - \nabla u_\varepsilon(t, y)\|_H \le L_{R, T}(t) \|x - y\|_H^\beta, \quad \forall x, y \in \mathcal{B}_R.
    \end{equation}
\end{enumerate}
\end{corollary}

\begin{proof}
The uniform bounds follow as a direct consequence of the penalty term's uniform boundedness. By Definition \ref{def:admissible_obstacle}, the admissibility of the dynamic obstacle $\Phi$ ensures that the family of penalty terms $\Pi_\varepsilon := \frac{1}{\varepsilon}(u_\varepsilon - \Phi)^+$ is uniformly bounded in the weighted parabolic topology. That is, there exists a constant $K > 0$, independent of $\varepsilon$, such that $\sup_{\varepsilon > 0} \|\Pi_\varepsilon\|_{\mathcal{B}_{m, T}} \le K$. Hence, the thesis follows directly from Lemma \ref{lem:global_reg_approx}.
\end{proof}

\subsection{Regularity of The Value Function}
\label{sec:weak_convergence}

We establish the convergence of the spatial gradients as $\varepsilon \downarrow 0$, deducing the $\mathcal{C}^{1,\beta}_{loc}(H)$ spatial regularity of the value function $V(t, x)$.

\begin{theorem}[Identification of the Fréchet Derivative and $\mathcal{C}^{1,\beta}_{loc}$ Regularity]
\label{thm:frechet_ident}
Let Assumptions \ref{ass:smoothing}, \ref{ass:coefficients} and \ref{ass:problem_data} hold. Let the dynamic obstacle $\Phi$ be admissible in the sense of Definition \ref{def:admissible_obstacle}. Let $u_\varepsilon$ be the unique mild solution of equation \eqref{eq:hjb_penalized} and $V$ be the pointwise limit of $u_\varepsilon$ as $\varepsilon \downarrow 0$. For any fixed $t \in [0, T)$ and $x\in H$, the following properties hold:
\begin{enumerate}[label={(\textit{\roman*})}, leftmargin=*]
    \item The family of penalized Fréchet gradients $\{\nabla u_{\varepsilon}(t, x)\}_{\varepsilon > 0}$ admits a subsequence weakly convergent to a limit $g(t, x) \in H$ as $\varepsilon\to0$.
    \item For any Hölder exponent $\beta \in (0, 1)$ satisfying $\gamma(1+\beta) < 1$, the weak limit $g(t, \cdot)$ is locally $\beta$-Hölder continuous on $H$.
    \item The mapping $V(t, \cdot)$ is Fréchet differentiable on $H$. Its spatial gradient coincides with the weak limit, yielding $\mathcal{C}^{1,\beta}_{loc}(H)$ spatial regularity.
\end{enumerate}
\end{theorem}

\begin{proof} The proof proceeds in three distinct parts, addressing each claim of the Theorem. \begin{enumerate}[label={(\textit{\roman*})}, leftmargin=*]
    \item \emph{Existence of the weak gradient limit.}
    Let $t \in [0, T)$. By Corollary \ref{cor:uniform_gradients}, $\{\nabla u_{\varepsilon}(t, \cdot)\}_{\varepsilon > 0}$ is uniformly bounded on bounded domains. Through Cantor diagonalization on a countable dense subset $\mathcal{D} \subset H$, there exists a subsequence $\varepsilon_n \downarrow 0$ such that:
    \begin{equation}
    \label{eq:proof_weak_limit_dense}
        \nabla u_{\varepsilon_n}(t, y) \rightharpoonup g(t, y) \quad \text{weakly in } H, \quad \forall y \in \mathcal{D}.
    \end{equation}
    For arbitrary $x, v \in H$ and $y \in \mathcal{D}$, the Cauchy-Schwarz inequality and the equi-Hölder estimates of Corollary \ref{cor:uniform_gradients} yield:
    \begin{equation}
    \label{eq:proof_cauchy_split}
        |\langle \nabla u_{\varepsilon_n}(t, x) - \nabla u_{\varepsilon_m}(t, x), v \rangle_H| \le 2 L_{R, T}(t) \|x - y\|_H^\beta \|v\|_H + |\langle \nabla u_{\varepsilon_n}(t, y) - \nabla u_{\varepsilon_m}(t, y), v \rangle_H|.
    \end{equation}
    By the density of $\mathcal{D}$, the first term on the right-hand side can be made arbitrarily small. Since $\{\nabla u_{\varepsilon_n}(t, y)\}$ converges weakly, the sequence of evaluations is Cauchy in $\mathbb{R}$. The weak sequential completeness of $H$ ensures the existence of a global mapping $g(t, \cdot): H \to H$ such that $\nabla u_{\varepsilon_n}(t, x) \rightharpoonup g(t, x)$ for all $x \in H$.

    \item \emph{Local equi-Hölder continuity.}
    For $x, y$ in a closed ball $B_R \subset H$, the weak lower semi-continuity of the norm and the uniform estimates in Corollary \ref{cor:uniform_gradients} imply:
    \begin{equation*}
        \|g(t, x) - g(t, y)\|_H \le \liminf_{n \to \infty} \|\nabla u_{\varepsilon_n}(t, x) - \nabla u_{\varepsilon_n}(t, y)\|_H \le L_{R, T}(t) \|x - y\|_H^\beta,
    \end{equation*}
    provided $\gamma(1+\beta) < 1$. Thus, $g(t, \cdot)$ inherits the local $\beta$-Hölder continuity.

    \item \emph{Identification of the Fréchet derivative.}
    For $x, h \in H$ and $\rho > 0$, the $\mathcal{C}^1$-regularity of $u_{\varepsilon_n}$ and the fundamental theorem of calculus yield:
    \begin{equation}
    \label{eq:proof_ftc_eps}
        u_{\varepsilon_n}(t, x + \rho h) - u_{\varepsilon_n}(t, x) = \int_0^\rho \langle \nabla u_{\varepsilon_n}(t, x + \theta h), h \rangle_H \, \mathrm{d}\theta.
    \end{equation}
    As $n \to \infty$, $u_{\varepsilon_n} \to V$ pointwise (Theorem \ref{thm:convergence_exact}) and $\langle \nabla u_{\varepsilon_n}, h \rangle_H \to \langle g, h \rangle_H$. The uniform bound $\|\nabla u_{\varepsilon_n}(t,\cdot)\|_H \le K_1(t) w_m$ justifies the Dominated Convergence Theorem, giving:
    \begin{equation*}
        V(t, x + \rho h) - V(t, x) = \int_0^\rho \langle g(t, x + \theta h), h \rangle_H \, \mathrm{d}\theta.
    \end{equation*}
    Dividing by $\rho$ and taking the limit as $\rho \to 0$, the continuity of $g$ establishes $\langle g(t, x), h \rangle_H$ as the Gâteaux derivative of $V$. Since a continuous Gâteaux derivative implies Fréchet differentiability \cite[Theorem 1.9]{AmPr95}, we identify $\nabla V \equiv g$. This, combined with the Hölder regularity in step (\textit{ii}), confirms $V(t, \cdot) \in \mathcal{C}^{1,\beta}_{loc}(H)$.
\end{enumerate}
\end{proof}
\section{Sufficient Conditions for Admissibility of Obstacles}\label{sec:admissible_obstacles}
The admissibility condition introduced in Definition \ref{def:admissible_obstacle} is the key estimate that allows the penalization procedure to be used beyond mere convergence. It provides a uniform bound on the positive constraint violation $u_\varepsilon-\Phi$, rescaled by the penalization parameter. This bound is needed in three places: first, in the joint continuity for the limit of the penalized solutions $u_\varepsilon$; second, in the verification argument identifying the limiting mild solution with the optimal stopping value; third, in the regularity analysis, where it yields $\varepsilon$-independent bounds on the gradients of the penalized solutions.

Since admissibility is formulated in terms of the unknown approximating family $u_\varepsilon$, it is not immediately suitable as a structural assumption on the data. The aim of this section is therefore to derive verifiable sufficient conditions on the obstacle $\Phi$ which imply admissibility. The main idea is to compare $\Phi$ with a mild supersolution of the penalized dynamics. This produces a uniform upper bound on the penalty term and shows that the abstract admissibility condition is satisfied in concrete and natural classes of obstacles.

\subsection{Uniform Inequality}
\label{sec:lewy_stampacchia}

We establish here a parabolic uniform inequality, providing a bound on the penalty term that is independent of the penalization parameter.

We assume that the dynamic obstacle is a mild supersolution of a suitable family of equations.
\begin{assumption}[Dynamic Obstacle Structure] 
\label{ass:dynamic_obstacle}
We assume that there exist $\eta \in \mathcal{B}_m(H)$ and $\xi \in \mathcal{B}_m([0, T] \times H)$ such that, for any arbitrary penalization parameter $\varepsilon > 0$, the dynamic obstacle $\Phi \in Lip_m([0, T] \times H)$ satisfies the shifted mild integral inequality:
\begin{equation} 
\label{eq:shifted_mild_obstacle}
    \Phi(t, x) \ge e^{-\frac{T-t}{\varepsilon}} P_{t,T}[\eta](x) + \int_t^T e^{-\frac{s-t}{\varepsilon}} P_{t,s} \left[ \frac{1}{\varepsilon} \Phi(s, \cdot) + \xi(s, \cdot) \right](x) \, \mathrm{d}s,
\end{equation}
for all $(t, x) \in [0, T] \times H$.
\end{assumption}

Under this assumption, we can derive a uniform, parameter-independent bound for the penalty term.
\begin{theorem}[Uniform Estimate] 
\label{thm:uniform_lewy_stampacchia}
Let Assumptions \ref{ass:coefficients}, \ref{ass:problem_data}, and \ref{ass:dynamic_obstacle} hold. Let $u_\varepsilon$ denote the unique mild solution of equation \eqref{eq:hjb_penalized}. If the terminal boundary compatibility condition $\Psi(x) \le \eta(x)$ is satisfied for all $x \in H$, then the penalty term $\Pi_\varepsilon$ defined in \eqref{eq:yosida_penalty} satisfies the pointwise estimate:
\begin{equation} 
\label{eq:lewy_stampacchia_bound}
    0 \le \Pi_\varepsilon(t, x) \le \|(f - \xi)^+\|_{\mathcal{B}_{m, T}} K_m w_m(x), \quad \forall (t, x) \in [0, T] \times H,
\end{equation}
where $K_m > 0$ is the constant derived in Lemma \ref{lem:moment_estimates}. This bound is entirely independent of $\varepsilon$.
\end{theorem}

\begin{proof}
Let $w_\varepsilon(t, x) := u_\varepsilon(t, x) - \Phi(t, x)$. Applying the shifted mild formulation (Theorem \ref{thm:equivalence_mild}) with $L := 1/\varepsilon$ and the dynamic sub-solution inequality for $\Phi$ yields:
\begin{equation*}
    w_\varepsilon(t, x) \le e^{-\frac{T-t}{\varepsilon}} P_{t,T} [\Psi - \eta](x) + \int_t^T e^{-\frac{s-t}{\varepsilon}} P_{t,s} \left[ f(s, \cdot) - \xi(s, \cdot) + \frac{1}{\varepsilon} w_\varepsilon(s, \cdot) - \Pi_\varepsilon(s, \cdot) \right](x) \, \mathrm{d}s.
\end{equation*}
Since $\frac{1}{\varepsilon} w_\varepsilon - \Pi_\varepsilon = -\frac{1}{\varepsilon} w_\varepsilon^- \le 0$, the positivity of $P_{t,s}$ allows discarding this non-positive component:
\begin{equation*}
    w_\varepsilon(t, x) \le e^{-\frac{T-t}{\varepsilon}} P_{t,T} [\Psi - \eta](x) + \int_t^T e^{-\frac{s-t}{\varepsilon}} P_{t,s} [ f(s, \cdot) - \xi(s, \cdot) ](x) \, \mathrm{d}s.
\end{equation*}
By the subadditivity and convexity of the positive part, Jensen's inequality ensures:
\begin{equation*}
    w_\varepsilon^+(t, x) \le e^{-\frac{T-t}{\varepsilon}} P_{t,T} [(\Psi - \eta)^+](x) + \int_t^T e^{-\frac{s-t}{\varepsilon}} P_{t,s} [(f - \xi)^+](x) \, \mathrm{d}s.
\end{equation*}
The terminal admissibility constraint $\Psi \le \eta$ implies $(\Psi - \eta)^+ \equiv 0$. Dividing by $\varepsilon$ to recover $\Pi_\varepsilon$, and invoking the polynomial growth stability $P_{t,s}[w_m] \le K_m w_m$ (Lemma \ref{lem:moment_estimates}), we deduce:
\begin{equation*}
    \Pi_\varepsilon(t, x) \le \int_t^T \frac{1}{\varepsilon} e^{-\frac{s-t}{\varepsilon}} P_{t,s} [(f - \xi)^+](x) \, \mathrm{d}s \le \|(f - \xi)^+\|_{\mathcal{B}_{m, T}} K_m w_m(x) \int_t^T \frac{1}{\varepsilon} e^{-\frac{s-t}{\varepsilon}} \, \mathrm{d}s.
\end{equation*}
Evaluating the integral and using $1 - e^{-\frac{T-t}{\varepsilon}} \le 1$ establishes the uniform bound:
\begin{equation*}
    \Pi_\varepsilon(t, x) \le \|(f - \xi)^+\|_{\mathcal{B}_{m, T}} K_m w_m(x).
\end{equation*}
This confirms that the penalization is bounded independently of $\varepsilon$, concluding the proof.
\end{proof}

\subsection{Regular Obstacles}
\label{sec:global_estimates}
In this subsection we show that sufficiently regular dynamic obstacles satisfy Assumption \ref{ass:dynamic_obstacle}. Consequently, they are admissible obstacles in the sense of Definition \ref{def:admissible_obstacle}.

\begin{assumption}[Regular obstacles]
\label{ass:regular_obstacles}
We assume that the dynamic obstacle $\Phi \in Lip_{m}([0,T]\times H)$ satisfies the following conditions:
\begin{enumerate}[label={(\textit{\roman*})}, leftmargin=*]
    \item \emph{Regularity.} The mapping $\Phi$ belongs to the class $\mathcal{C}^{1,2}([0, T] \times H; \mathbb{R})$.
    \item \emph{Adjoint Compatibility.} For any $(t, x) \in [0, T] \times H$, the spatial Fréchet gradient maps into the domain of the adjoint generator, namely $\nabla \Phi(t, x) \in \mathrm{Dom}(A^*)$. Furthermore, the mapping $(t, x) \mapsto A^* \nabla \Phi(t, x)$ is continuous from $[0, T] \times H$ into $H$.
    \item \emph{Polynomial Growth.} There exists a constant $K > 0$ such that:
    $$|\partial_t \Phi(t, x)| + \|\nabla \Phi(t, x)\|_H + \|\nabla^2 \Phi(t, x)\|_{\mathcal{L}(H)} + \|A^* \nabla \Phi(t, x)\|_H \le K w_m(x),$$
    for all $(t, x) \in [0, T] \times H$.
\end{enumerate}
\end{assumption}
Under Assumption \ref{ass:regular_obstacles}, the differential terms are well-defined, continuous, and belong to $\mathcal{B}_m([0, T] \times H)$. Consequently, by the mild Itô formula (see, e.g., \cite[Proposition 1.165]{FaGoSw17}), the obstacle satisfies the following identity for all $s \in [t, T]$:
\begin{equation}
\label{eq:mild_ito_identity}
\begin{aligned}
    P_{t,s}[\Phi(s, \cdot)](x) - \Phi(t, x) &= \int_t^s P_{t,r}\Big[ \partial_r \Phi(r, \cdot) + \langle \cdot, A^* \nabla \Phi(r, \cdot) \rangle_H \\
    & \quad + \langle b(r, \cdot), \nabla \Phi(r, \cdot) \rangle_H + \frac{1}{2} \Tr \big( \sigma(r, \cdot) \sigma^*(r, \cdot) \nabla^2 \Phi(r, \cdot) \big) \Big](x) \, \mathrm{d}r.
\end{aligned}
\end{equation}
This identity allow us to prove the following Lemma.
\begin{lemma}[Mild Variation of Constants]
\label{lem:mild_variation_constants}
Let Assumption \ref{ass:regular_obstacles} hold. For any penalization parameter $\varepsilon > 0$ and any initial state $(t, x) \in [0, T] \times H$, the continuous obstacle $\Phi$ satisfies the following identity:
\begin{equation}
\label{eq:mild_variation_constants}
\begin{aligned}
    \Phi(t, x) &= e^{-\frac{T-t}{\varepsilon}} P_{t,T}[\Phi(T, \cdot)](x) + \int_t^T e^{-\frac{s-t}{\varepsilon}} P_{t,s} \Big[ -\partial_s \Phi(s, \cdot) - \langle \cdot, A^* \nabla \Phi(s, \cdot) \rangle_H \\
    & \quad- \langle b(s, \cdot), \nabla \Phi(s, \cdot) \rangle_H - \frac{1}{2} \Tr \big( \sigma(s, \cdot) \sigma^*(s, \cdot) \nabla^2 \Phi(s, \cdot) \big) + \frac{1}{\varepsilon} \Phi(s, \cdot) \Big](x) \, \mathrm{d}s.
\end{aligned}
\end{equation}
\end{lemma}

\begin{proof}
Fix an initial state $(t, x) \in [0, T) \times H$ and an arbitrary penalization parameter $\varepsilon > 0$. By Assumption \ref{ass:regular_obstacles}, the mapping $\Phi$ and all its spatial and temporal differential components exhibit at most polynomial growth. The mild Itô identity \eqref{eq:mild_ito_identity} guarantees that the evaluation mapping $s \mapsto P_{t,s}[\Phi(s, \cdot)](x)$ is absolutely continuous. Applying the product rule for absolutely continuous functions to the discounted mapping $s \mapsto e^{-\frac{s-t}{\varepsilon}} P_{t,s}[\Phi(s, \cdot)](x)$ over the interval $[t, T]$ and integrating we get:
\begin{equation*}
\begin{aligned}
    e^{-\frac{T-t}{\varepsilon}} P_{t,T}[\Phi(T, \cdot)](x) - \Phi(t, x)
    &= \int_t^T e^{-\frac{s-t}{\varepsilon}} P_{t,s} \Big[ -\frac{1}{\varepsilon} \Phi(s, \cdot) + \partial_s \Phi(s, \cdot) + \langle \cdot, A^* \nabla \Phi(s, \cdot) \rangle_H \\
    & \quad + \langle b(s, \cdot), \nabla \Phi(s, \cdot) \rangle_H + \frac{1}{2} \Tr \big( \sigma(s, \cdot) \sigma^*(s, \cdot) \nabla^2 \Phi(s, \cdot) \big) \Big](x) \, \mathrm{d}s.
\end{aligned}
\end{equation*}
A direct algebraic rearrangement recovers \eqref{eq:mild_variation_constants}, concluding the proof.
\end{proof}

\subsection{Square Norm}
\label{sec:square_norm}
\noindent We show that the square norm is an admissible obstacle in the sense of Definition \ref{def:admissible_obstacle}. Since the gradient of the square norm does not generally belong in the domain of the adjoint operator, is not possible to apply directly the mild Itô formula. Hence, we regularize the stochastic flow through the Yosida approximants, we evaluate the explicit Itô differential terms, extract uniform bounds, and recover the mild supersolution property .

\begin{theorem}[Quadratic Mild Supersolutions]
\label{thm:quadratic_supersolutions}
Let Assumption \ref{ass:coefficients} hold. Assume, moreover, that the unbounded linear operator $A$ satisfies $\langle Ax, x \rangle_H \le \omega \|x\|_H^2$ for some $\omega \in \mathbb{R}$ and all $x \in \mathrm{Dom}(A)$. Let $\phi(x) := \|x\|_H^2$. Then, there exists a constant $C > 0$, such that for any penalization parameter $\varepsilon > 0$ and initial time $t \in [0, T]$, $\phi$ satisfies:
\begin{equation} 
\label{eq:quadratic_supersolution}
    \phi(x) \ge e^{-\frac{T-t}{\varepsilon}} P_{t,T} [\phi](x) + \int_t^T e^{-\frac{s-t}{\varepsilon}} P_{t,s} \left[ \frac{1}{\varepsilon} \phi - C(1+\phi) \right](x) \, \mathrm{d}s.
\end{equation}
\end{theorem}

\begin{proof}
The spatial Fréchet gradient of the quadratic function is $\nabla \phi(x) = 2x$. Since arbitrary states $x \in H$ do not necessarily belong to $\mathrm{Dom}(A^*)$, we introduce the classical Yosida approximants $A_\lambda := \lambda A(\lambda I - A)^{-1}$ for $\lambda > 0$ (see, e.g., \cite[Chapter 1, Section 1.3]{Pa83}). Since $A_\lambda$ is a bounded linear operator on $H$, we can apply the mild Itô formula to the discounted process $e^{-\frac{s-t}{\varepsilon}} \phi(X_\lambda(s))$, obtaining:
\begin{equation}\label{provaprova}
\begin{aligned}
    \E\left[ e^{-\frac{T-t}{\varepsilon}} \phi(X_\lambda(T)) \right] - \phi(x) &= \E \Bigg[ \int_t^T e^{-\frac{s-t}{\varepsilon}} \Big( -\frac{1}{\varepsilon} \phi(X_\lambda(s)) + 2\langle A_\lambda X_\lambda(s), X_\lambda(s) \rangle_H \\
    & \quad + 2\langle b(s, X_\lambda(s)), X_\lambda(s) \rangle_H + \operatorname{Tr}\big(\sigma(s, X_\lambda(s))\sigma^*(s, X_\lambda(s))\big) \Big) \mathrm{d}s \Bigg].
\end{aligned}
\end{equation}
By the properties of Yosida approximants, the dissipativity of $A$ is preserved, ensuring $2\langle A_\lambda y, y \rangle_H \le 2\omega \|y\|_H^2$ for all $y \in H$ uniformly in $\lambda$. Furthermore, the linear growth bounds of the coefficients (Assumption \ref{ass:coefficients}) imply, via the Cauchy-Schwarz and Young inequalities:
\begin{equation*}
    2\langle b(s, y), y \rangle_H \le 2K^2 + (2K^2 + 1)\|y\|_H^2, \quad \operatorname{Tr}(\sigma(s, y)\sigma^*(s, y)) \le 2K^2(1 + \|y\|_H^2).
\end{equation*}
Hence, for sufficiently large $C>0$, we get:
\begin{equation*}
    2\langle A_\lambda y, y \rangle_H + 2\langle b(s, y), y \rangle_H + \operatorname{Tr}\big(\sigma(s, y)\sigma^*(s, y)\big) \le C(1 + \|y\|_H^2).
\end{equation*}
Substituting this bound into \ref{provaprova} gives:
\begin{equation}\label{provaprova1}
    \phi(x) \ge \E\left[ e^{-\frac{T-t}{\varepsilon}} \phi(X_\lambda(T)) \right] - \E \left[ \int_t^T e^{-\frac{s-t}{\varepsilon}} \Big( -\frac{1}{\varepsilon} \phi(X_\lambda(s)) + C(1 + \phi(X_\lambda(s))) \Big) \mathrm{d}s \right].
\end{equation}
As $\lambda \to \infty$ we have that $\E[\|X_\lambda(s) - X(s)\|_H^2] \to 0$ uniformly on $[t, T]$ (see, e.g., \cite[Proposition 3.2]{GaLe11}). By the continuity of the norm and the polynomial growth bound, applying the Dominated Convergence Theorem to the temporal integral recovers the exact mild supersolution inequality \eqref{eq:quadratic_supersolution} for the true flow $X(s)$, concluding the proof.
\end{proof}

\subsection{Cylindrical Lipschitz Concave Functions}
\label{sec:cylindrical_obstacles}

\noindent We show that cylindrical obstacles defined by Lipschitz concave functions are admissible. Standard approximation arguments justify the mild Itô formula, while concavity neutralizes the second-order trace term, yielding the required uniform bound.

\begin{theorem}[Estimate for Projected Payoffs] 
\label{thm:lewy_stampacchia_projected}
Let Assumptions \ref{ass:coefficients} and \ref{ass:problem_data} hold. Define the dynamic obstacle $\Phi(x) := g(\langle h, x \rangle_H)$, where $g: \R \to \R$ is concave and $L$-Lipschitz, and $h \in \mathrm{Dom}(A^*)$. Let $u_\varepsilon$ be the unique mild solution to \eqref{eq:hjb_penalized}. Then, the penalty term $\Pi_\varepsilon$ defined in \eqref{eq:yosida_penalty} satisfies:
\begin{equation}
\label{eq:ls_projected_bound}
    0 \le \Pi_\varepsilon(t, x) \le C w_m(x), \quad \forall (t,x) \in [0,T] \times H,
\end{equation}
for a constant $C > 0$ independent of $\varepsilon$. Consequently, $\Phi$ is an admissible obstacle in the sense of Definition \ref{def:admissible_obstacle}.
\end{theorem}

\begin{proof}
Let $g_n \in \mathcal{C}^\infty(\R)$ be a standard sequence of smooth, concave regularizations converging uniformly to $g$. We can choose these approximants such that $g_n'' \le 0$, $\sup_{n \in \mathbb{N}} \|g_n'\|_\infty \le L < \infty$, and $\|g_n''\|_\infty < \infty$ for all $n \in \mathbb{N}$. Define the regularized cylindrical obstacles $\Phi_n(x) := g_n(\langle h, x \rangle_H)$. Since $h \in \mathrm{Dom}(A^*)$, the Fréchet gradient $\nabla \Phi_n(x) = g_n'(\langle h, x \rangle_H) h$ maps into $\mathrm{Dom}(A^*)$ for all $x \in H$. Thus, $\Phi_n$ satisfies Assumption \ref{ass:regular_obstacles}. Moreover, we can evaluate explicitly:
\begin{equation*}
\begin{aligned}
    &\langle x, A^* \nabla \Phi_n(x) \rangle_H + \langle b(t, x), \nabla \Phi_n(x) \rangle_H + \frac{1}{2} \Tr \big( \sigma(t, x) \sigma^*(t, x) \nabla^2 \Phi_n(x) \big) \\
    &= g_n'(\langle h, x \rangle_H) \big( \langle x, A^* h \rangle_H + \langle b(t, x), h \rangle_H \big) + \frac{1}{2} g_n''(\langle h, x \rangle_H) \|\sigma^*(t, x)h\|_H^2.
\end{aligned}
\end{equation*}
Using the bounds $g_n'' \le 0$ and $\sup_{n \in \mathbb{N}}\|g_n'\|_\infty \le L$ together with the linear growth of $b$, there exists a constant $C_1 > 0$, independent of $n$, such that:
\begin{equation*}
    \langle x, A^* \nabla \Phi_n(x) \rangle_H + \langle b(t, x), \nabla \Phi_n(x) \rangle_H + \frac{1}{2} \Tr \big( \sigma(t, x) \sigma^*(t, x) \nabla^2 \Phi_n(x) \big) \le C_1 (1 + \|x\|_H).
\end{equation*}
By Lemma \ref{lem:mild_variation_constants} and Theorem \ref{thm:uniform_lewy_stampacchia}, the penalty term $\Pi_{\varepsilon, n} := \frac{1}{\varepsilon}(u_{\varepsilon, n} - \Phi_n)^+$ satisfies:
\begin{equation*}
    0 \le \Pi_{\varepsilon, n}(t, x) \le \Big( C_1 (1 + \|x\|_H) + \|f\|_{\mathcal{B}_{m, T}} w_m(x) \Big)^+.
\end{equation*}
Since $1 + \|x\|_H \le \sqrt{2}w_m(x)$, we find a constant $C > 0$, independent of both $\varepsilon$ and $n$, such that $\Pi_{\varepsilon, n}(t, x) \le C w_m(x)$. As $n \to \infty$, the stability of mild solutions (Lemma \ref{lem:stability_mild}) implies $u_{\varepsilon, n} \to u_\varepsilon$ in $\mathcal{B}_m([0, T] \times H)$. The uniform convergence $\Phi_n \to \Phi$ ensures $\Pi_{\varepsilon, n} \to \Pi_\varepsilon$ pointwise, preserving the bound:
\begin{equation*}
    \Pi_\varepsilon(t, x) \le C w_m(x).
\end{equation*}
This estimate confirms admissibility, concluding the proof.
\end{proof}
\section{Application: The Stochastic Heat Equation}
\label{sec:example_heat}

\noindent This section validates the abstract framework by applying the continuous-differentiability result to a linear stochastic heat equation. We provide an explicit construction of a trace-class covariance operator that satisfies the null controllability condition, ensuring that the transition semigroup exhibits the regularizing properties required for the spatial differentiability of the value function.

Consider the spatial domain $\mathcal{O} := (0, \pi)$ and the state space $H := L^2(0, \pi)$. The infinitesimal generator $A := \Delta = \partial_{xx}$, with domain $\mathrm{Dom}(A) := H^2(0, \pi) \cap H^1_0(0, \pi)$, generates an analytic contraction semigroup $\{e^{tA}\}_{t \ge 0}$ (see, e.g., \cite[Chapter 1, Theorem 4.3]{Pa83}). The spectrum of $A$ consists of eigenvalues $-\lambda_k = -k^2$ corresponding to the orthonormal basis $e_k(x) = \sqrt{2/\pi} \sin(kx)$ for $k \in \mathbb{N}$. We analyze the following stochastic heat equation with additive noise:
\begin{equation}
\label{eq:heat_dynamics}
\begin{cases}
\mathrm{d}X(s) = AX(s) \mathrm{d}s + B \mathrm{d}W(s), \quad s \in (t, T], \\
X(t) = x \in H,
\end{cases}
\end{equation}
where $W$ is a cylindrical Wiener process on a separable Hilbert space $\Xi$ and $B \in \mathcal{L}(\Xi, H)$. Let $Q := BB^* \in \mathcal{L}_1(H)$ be the covariance operator and $Q_t := \int_0^t e^{sA} Q e^{sA^*} \, \mathrm{d}s$ the associated Gramian. 

To guarantee the spatial smoothing properties of the transition semigroup, we require the stochastic convolution to satisfy a certain range inclusion, commonly known in control theory as null controllability.
\begin{assumption}[Null Controllability]
\label{ass:null_controllability}
For any $t > 0$, the inclusion $\mathrm{Im}(e^{tA}) \subseteq \mathrm{Im}(Q_t^{1/2})$ holds.
\end{assumption}

\noindent The simultaneous requirement of a trace-class covariance and null controllability necessitates a precise decay rate for the eigenvalues of $Q$. We now demonstrate that it is possible to construct a valid noise covariance operator that balances the trace-class requirement with the necessary regularizing singularity.
\begin{lemma}[Construction of the Covariance Operator]
\label{lem:explicit_covariance}
There exists a positive, trace-class operator $Q \in \mathcal{L}_1(H)$ satisfying Assumption \ref{ass:null_controllability}. Moreover, the operator $\Gamma(t) := Q_t^{-1/2} e^{tA}$ is bounded with norm $\|\Gamma(t)\|_{\mathcal{L}(H)} \le C t^{-\gamma}$ for some $\gamma \in (0, 1)$.
\end{lemma}

\begin{proof}
Let $Q e_k = q_k e_k$ with $q_k := k^{-2\beta}$. The trace-class condition $\Tr(Q) = \sum k^{-2\beta} < \infty$ is satisfied for $\beta > 1/2$. The Gramian $Q_t$ is diagonal with eigenvalues $\mu_{k,t} = q_k (1 - e^{-2tk^2}) / (2k^2)$. The operator $\Gamma(t)$ is bounded if and only if
\begin{equation*}
\|\Gamma(t)\|_{\mathcal{L}(H)}^2 = \sup_{k \ge 1} \frac{e^{-2tk^2}}{\mu_{k,t}} = \sup_{k \ge 1} \frac{2 k^{2\beta+2} e^{-2tk^2}}{1 - e^{-2tk^2}} < \infty.
\end{equation*}
By the substitution $y := tk^2$, the supremum is bounded by $C t^{-(\beta+1)} \sup_{y>0} [y^{\beta+1} e^{-2y} (1 - e^{-2y})^{-1}]$. The mapping $y \mapsto y^{\beta+1} e^{-2y} (1 - e^{-2y})^{-1}$ is uniformly bounded on $(0, \infty)$, yielding $\|\Gamma(t)\|_{\mathcal{L}(H)} \le \sqrt{C} t^{-\frac{\beta+1}{2}}$. The integrability condition $\gamma := (\beta+1)/2 < 1$ requires $\beta < 1$. Thus, any $\beta \in (1/2, 1)$ ensures a trace-class covariance and an integrable singularity.
\end{proof}

\noindent Consider the optimal stopping problem for \eqref{eq:heat_dynamics} under the structural conditions on $(f, \Phi, \Psi)$ defined in Assumption \ref{ass:problem_data} and the admissibility of $\Phi$ (Definition \ref{def:admissible_obstacle}). The following lemma concludes that this system fits the abstract regularity framework, yielding a precise quantitative bound on the Hölder regularity of the value function's gradient.
\begin{lemma}[Smoothing and Regularity]
\label{lem:heat_smoothing}
Let $Q$ be as in Lemma \ref{lem:explicit_covariance} with $\beta \in (1/2, 1)$. Then the transition semigroup $P_{t,s}$ satisfies Assumption \ref{ass:smoothing} with $\gamma = (\beta+1)/2 \in (3/4, 1)$. Consequently, the value function $V$ is the unique exact mild solution to the parabolic variational inequality in the sense of Definition \ref{def:mild_subsolutions} and satisfies $V(t, \cdot) \in \mathcal{C}^{1,\beta_{loc}}_{loc}(H)$ for $t \in [0, T)$, where the Hölder exponent is constrained by $0 < \beta_{loc} < 1/\gamma - 1 < 1/3$.
\end{lemma}

\begin{proof}
By Assumption \ref{ass:null_controllability} we can apply \cite[Theorem 9.26]{DPZa14}, deducing that Assumption \ref{ass:smoothing} holds in the case of bounded continuous functions. The extension of this result to the weighted polynomial space $\mathcal{C}_m(H)$ can be done as in \cite[Theorem 4.2]{Ce95}. Hence, Assumption \ref{ass:smoothing} holds with $\gamma = (\beta+1)/2$. The identification of $V$ and its regularity follow directly from Theorems \ref{thm:optimal_strategy} and \ref{thm:frechet_ident}.
\end{proof}

\bibliographystyle{plain}
\bibliography{bigbib}
\end{document}